\documentclass[12pt,english]{article}
\usepackage[leqno]{amsmath}

\usepackage[T1]{fontenc}
\usepackage[latin9]{inputenc}
\usepackage{geometry}
\usepackage{mathtools}
\mathtoolsset{showonlyrefs=false}
\usepackage{microtype}

\usepackage[leqno]{amsmath}

\makeatletter
\newcommand{\leqnomode}{\tagsleft@true\let\veqno\@@leqno}
\newcommand{\reqnomode}{\tagsleft@false\let\veqno\@@eqno}
\reqnomode  

\usepackage{amsthm}
\usepackage{amssymb}

\numberwithin{equation}{section}
\numberwithin{figure}{section}
\theoremstyle{plain}
\newtheorem{thm}{\protect\theoremname}[section]
  \theoremstyle{plain}
  \newtheorem{lem}[thm]{\protect\lemmaname}
  \theoremstyle{plain}
  \newtheorem{cor}[thm]{\protect\corollaryname}
  \theoremstyle{definition}

\usepackage{bm}

\usepackage{hyperref}

\usepackage{colortbl}
\usepackage{xparse}
\usepackage{caption}

\usepackage{titlesec}
\titleformat{\section}
  {\normalfont\large\bfseries}{\thesection.}{.8ex}{} 
\titleformat{\subsection}
  {\normalfont\bfseries}{\thesubsection.}{.8ex}{} 

\let\originalleft\left
\let\originalright\right
\renewcommand{\left}{\mathopen{}\mathclose\bgroup\originalleft}
\renewcommand{\right}{\aftergroup\egroup\originalright}

\usepackage{authblk}
\title{Plethystic formulas for permutation enumeration}
\author[1]{Ira M.\ Gessel\thanks{\tt{gessel@brandeis.edu}\\
\null\hskip17pt Supported by a grant from the Simons Foundation (\#427060, Ira Gessel). }}
\author[2]{Yan Zhuang\thanks{\tt{yazhuang@davidson.edu}}}
\affil[1]{Department of Mathematics, Brandeis University} 
\affil[2]{Department of Mathematics and Computer Science, Davidson College}
\date{\vspace{-0.15in}\today}

\makeatother

\usepackage{babel}
  \providecommand{\corollaryname}{Corollary}
  \providecommand{\lemmaname}{Lemma}
  \providecommand{\problemname}{Problem}
\providecommand{\theoremname}{Theorem}

\usepackage[normalem]{ulem} 
\usepackage[dvipsnames]{xcolor}
\newcommand{\commentcolor}{OrangeRed}

\definecolor{light-gray}{gray}{0.6}

\let\oldthefootnote\thefootnote
\newcommand\cfootnote[1]{%
  \def\thefootnote{\color{\commentcolor}\oldthefootnote}%
  \footnote{\color{\commentcolor}#1}%
    \let\thefootnote\oldthefootnote}

\AtBeginDocument{%
   \def\MR#1{}
}

\begin{document}
\global\long\def\Des{\operatorname{Des}}
\global\long\def\pk{\operatorname{pk}}
\global\long\def\lpk{\operatorname{lpk}}
\global\long\def\val{\operatorname{val}}
\global\long\def\Comp{\operatorname{Comp}}
\global\long\def\udr{\operatorname{udr}}
\global\long\def\des{\operatorname{des}}
\global\long\def\dasc{\operatorname{dasc}}
\global\long\def\ddes{\operatorname{ddes}}
\global\long\def\lddes{\operatorname{lddes}}
\global\long\def\br{\operatorname{br}}
\global\long\def\maj{\operatorname{maj}}
\global\long\def\st{\operatorname{st}}
\global\long\def\inv{\operatorname{inv}}
\global\long\def\fix{\operatorname{fix}}
\global\long\def\ps{\operatorname{ps}}
\global\long\def\GL{\operatorname{GL}}

\maketitle
\begin{abstract}
We prove several general formulas for the distributions of various permutation statistics over any set of permutations whose quasisymmetric generating function is a symmetric function. Our formulas involve certain kinds of plethystic substitutions on quasisymmetric generating functions, and the permutation statistics we consider include the descent number, peak number, left peak number, and the number of up-down runs. We apply these results to cyclic permutations, involutions, and derangements, and more generally, to derive formulas for counting all permutations by the above statistics jointly with the number of fixed points and jointly with cycle type.
A number of known formulas are recovered as special cases of our results, including formulas of D\'esarm\'enien--Foata, Gessel\textendash Reutenauer, Stembridge, Fulman, Petersen, Diaconis\textendash Fulman\textendash Holmes, Zhuang, and Athanasiadis.
\end{abstract}
\textbf{\small{}Keywords:}{\small{} permutation statistics, descents,
peaks, cycle type, symmetric functions, plethysm}{\let\thefootnote\relax}

\tableofcontents{}

\section{Introduction}
\label{s-intro}
Let $\pi=\pi(1)\pi(2)\cdots\pi(n)$ be an element of the symmetric
group $\mathfrak{S}_{n}$ of permutations of the set $[n]\coloneqq\{1,2,\dots,n\}$.
We say that $i\in[n-1]$ is a \textit{descent} of $\pi$ if $\pi(i)>\pi(i+1)$.
Let $\des(\pi)$ denote the number of descents of $\pi$ and let $\maj(\pi)$
denote the sum of all descents of $\pi$. The \textit{descent number}
$\des$ and \textit{major index} $\maj$ are classical permutation
statistics whose study dates back to Percy MacMahon \cite{macmahon}.

The distribution of the descent number over $\mathfrak{S}_{n}$ 
for $n\ge1$ is encoded by the $n$th \textit{Eulerian polynomial}\footnote{In some works, $A_n(t)$ is instead defined to be $\sum_{\pi\in\mathfrak{S}_{n}}t^{\des(\pi)}$.}
\[
A_{n}(t)\coloneqq\sum_{\pi\in\mathfrak{S}_{n}}t^{\des(\pi)+1},
\]
and the joint distribution of the descent number and major index by
the $n$th $q$-\textit{Eulerian polynomial }
\[
A_{n}(q,t)\coloneqq\sum_{\pi\in\mathfrak{S}_{n}}q^{\maj(\pi)}t^{\des(\pi)+1},
\]
with $A_0(t) = A_0(q,t) =1$. MacMahon \cite[Vol. 2, Section IX]{macmahon} proved a formula\footnote{See \cite[Section 3]{Gessel2016} for more on MacMahon's work.} of which a special case is
\begin{equation}
\frac{A_{n}(q,t)}{(1-t)(1-tq)\cdots(1-tq^{n})}=\sum_{k=0}^{\infty}[k]_{q}^{n}t^{k},\label{e-qeulerian}
\end{equation}
where
\[
[k]_{q}\coloneqq1+q+q^{2}+\cdots+q^{k-1}.
\]
(This formula is often attributed to Carlitz \cite{carlitz1975}.)
Note that (\ref{e-qeulerian}) reduces to the classical identity 
\begin{equation}
\frac{A_{n}(t)}{(1-t)^{n+1}}=\sum_{k=0}^{\infty}k^{n}t^{k}\label{e-eulerian}
\end{equation}
(which is sometimes used as the definition for Eulerian polynomials). 

Equation (\ref{e-qeulerian}) allows one to compute the joint distribution of the descent number and major index over $\mathfrak{S}_{n}$, but one may also want to study the joint distribution of these statistics, and others, over certain interesting subsets of $\mathfrak{S}_{n}$. 

Our work is concerned with \emph{descent statistics}: permutation statistics that are determined by the descent set. Examples of descent statistics include the descent number and major index, as well as the peak number, the left peak number, and the number of up-down runs, which we will study in this paper. Quasisymmetric functions---which are certain power series in infinitely many variables that generalize symmetric functions---encode descent sets, and in many cases, generating functions for descent statistics can be extracted in a useful way from quasisymmetric functions. For example, 
if $\Pi$ is a subset of $\mathfrak{S}_{n}$
with quasisymmetric generating function $Q(\Pi)$  (see Section \ref{ss-qsymgf} for relevant definitions),
then 
\begin{align}
\frac{\sum_{\pi\in\Pi}t^{\des(\pi)+1}q^{\maj(\pi)}}{(1-t)(1-tq)\cdots(1-tq^{n})} & =\sum_{k=0}^{\infty}\phi_{k}(Q(\Pi))t^{k}\label{e-desmaj}
\end{align}
where $\phi_{k}(f)\coloneqq f(q^{k-1},q^{k-2},\dots,1)$.  (See \cite[Section 4]{Gessel1984}.)
If $Q(\Pi)$ is symmetric, then we can use operations on symmetric functions to extract generating function formulas for permutation statistics. 

In 1993, the first author and Christophe Reutenauer \cite{Gessel1993} proved that the set of permutations with a prescribed cycle type has a symmetric quasisymmetric generating function, and used this fact and (\ref{e-desmaj}) to derive formulas for the joint distribution of
$\des$ and $\maj$ over cyclic permutations, involutions, and derangements. Later, Jason Fulman \cite[Theorem 1]{Fulman1998} used the work of Gessel and Reutenauer to derive a formula for the joint distribution of $\des$, $\maj$,
and cycle type over $\mathfrak{S}_{n}$. These results are notable
because the descent number and major index are statistics which encode
properties of a permutation in one-line representation, 
and one would not expect such statistics to be compatible with the cycle structure. See \cite{Diaconis1995, Elizalde2011, Elizalde2019, Fulman2001, Han2009, Moustakas2020, Poirier1998, Reiner1993, Stanley2007, Thibon2001} for related work.

Given $\pi\in\mathfrak{S}_{n}$, we say that $i\in\{2,\dots,n-1\}$ is a \textit{peak} of $\pi$ if $\pi(i-1)<\pi(i)>\pi(i+1)$, and that $i\in[n-1]$ is a \textit{left peak} of $\pi$ if $i$ is a peak or if $i=1$ and $\pi(1)>\pi(2)$. A \textit{birun}\footnote{Biruns are also commonly called ``alternating runs''.} of $\pi$ is a maximal monotone consecutive subsequence, and an up-down run of $\pi$ is a birun of $\pi$ or the letter $\pi(1)$ if $\pi(1)>\pi(2)$. Let $\pk(\pi)$ denote the number of peaks of $\pi$, $\lpk(\pi)$ the number of left peaks of $\pi$, and $\udr(\pi)$ the number of up-down runs\footnote{We note that $\udr(\pi)$ is also equal to the length of the longest alternating subsequence of $\pi$, which was studied in depth by Stanley \cite{Stanley2008}.} of $\pi$. For example, if $\pi=71462853$, the peaks of $\pi$ are $4$ and $6$; the left peaks of $\pi$ are $1$, $4$, and $6$; and the up-down runs of $\pi$ are $7$, $71$, $146$, $62$, $28$, and $853$. Thus, $\pk(\pi)=2$, $\lpk(\pi)=3$, and $\udr(\pi)=6$. 

The \textit{peak number} $\pk$, the \textit{left peak number} $\lpk$, and the \textit{number of up-down runs} $\udr$ are\textemdash like the descent number and major index\textemdash statistics which deal with permutations in one-line representation. All of these statistics have been well-studied (see, e.g., \cite{Gessel2018,Zhuang2016,Zhuang2017} by the present authors, and the references therein), but to our knowledge the only known result concerning a distribution of any of these statistics while refining by cycle structure is due to Diaconis, Fulman, and Holmes, who derived a formula \cite[Corollary 3.8]{Diaconis2013} for the joint distribution of the peak number and cycle type over $\mathfrak{S}_{n}$ in their analysis of casino shelf shuffling machines.

The main results of our paper are general formulas, analogous to \eqref{e-desmaj}, which can be used to study the joint distribution of $\pk$ and $\des$, the joint distribution of $\lpk$ and $\des$, and the distribution of $\udr$ over any subset of $\mathfrak{S}_{n}$ whose quasisymmetric generating function is symmetric. These formulas involve \textit{plethysm}, an operation on symmetric functions originating in the representation theory of the general linear groups $\GL_{n}(\mathbb{C})$ and the symmetric groups $\mathfrak{S}_{n}$. In recent decades, plethysm has been extended to more general formal power series rings and has found numerous applications within algebraic combinatorics, e.g., in the theory of Macdonald polynomials. See Macdonald \cite[Chapter I, Section 8]{Macdonald1995}, Stanley \cite[Chapter 7, Appendix 2]{Stanley2001}, Haglund \cite[Chapter 1]{Haglund2008}, and Loehr\textendash Remmel \cite{Loehr2011} for introductory references on plethysm.

The structure of our paper is as follows. In Section 2, we review some relevant results from the basic theory of symmetric and quasisymmetric functions, give an introductory account of plethystic calculus, prove some preliminary lemmas involving plethystic substitutions, and give ribbon and power sum expansions of several symmetric function identities which are relevant for the study of peaks, left peaks, and up-down runs. We prove our main results in Section 3. The rest of the paper presents applications of our main results: Section 4 concerns cyclic permutations, Section 5 concerns involutions, Section 6 concerns derangements and (more generally) counting permutations by fixed points, and Section 7 concerns counting permutations by cycle type. We recover as special cases the formulas of Fulman and Diaconis\textendash Fulman\textendash Holmes mentioned above\textemdash as well as various formulas by D\'esarm\'enien--Foata, Gessel\textendash Reutenauer, Stembridge, Petersen, Zhuang, and Athanasiadis\textemdash but most of our results are new and pertain to permutation statistic distributions that have been previously unstudied. We conclude in Section 8 with a brief discussion of future work.

\section{Preliminaries}

\subsection{Review of basic symmetric function theory}

We assume familiarity with basic definitions from the theory of symmetric
functions at the level of Stanley \cite[Chapter 7]{Stanley2001}.
In this section we establish notation and review some
elementary facts which will be needed for our work. See also Macdonald
\cite[Chapter 1]{Macdonald1995}, Sagan \cite[Chapter 4]{Sagan2001},
and Grinberg\textendash Reiner \cite[Section 2]{Grinberg2020} for
other treatments of basic symmetric function theory.

We use the notations $\lambda\vdash n$ and $\left|\lambda\right|=n$
to indicate that $\lambda$ is a partition of $n$, and we let $l(\lambda)$
denote the number of parts of $\lambda$. We write $\lambda=(1^{m_{1}}2^{m_{2}}\cdots)$
to mean that $\lambda$ has $m_{1}$ parts of size 1, $m_{2}$ parts
of size 2, and so on; alternatively, we write $\lambda=(\lambda_{1},\lambda_{2},\dots,\lambda_{r})$
to mean that $\lambda$ has parts $\lambda_{1},\lambda_{2},\dots,\lambda_{r}$.
For example, $\text{\ensuremath{\lambda}=}(1^{2}3^{4}4^{1})$ and
$\lambda=(4,3,3,3,3,1,1)$ denote the same partition, and here we
have $\left|\lambda\right|=18$ and $l(\lambda)=7$.

Let $\Lambda$ denote the $\mathbb{Q}$-algebra of symmetric functions in the variables $x_{1},x_{2},\dots$. We recall the important bases for $\Lambda$: the monomial symmetric functions $m_{\lambda}$, the complete symmetric functions $h_{\lambda}$, the elementary symmetric
functions $e_{\lambda}$, the power sum symmetric functions $p_{\lambda}$, and the Schur functions $s_{\lambda}$. (As usual, we write $h_{(n)}$
as $h_{n}$, $e_{(n)}$ as $e_{n}$, and $p_{(n)}$ as $p_{n}$.)

We will also work with symmetric functions with coefficients involving additional variables such as $t$, $y$, $z$, and $\alpha$, and symmetric functions of unbounded degree like $H(z)$, defined below.

Define $H(z)\coloneqq\sum_{n=0}^{\infty}h_{n}z^{n}$ and $E(z)\coloneqq\sum_{n=0}^{\infty}e_{n}z^{n}$
to be the ordinary generating functions for the  $h_{n}$ and $e_{n}$, respectively.
It is well known \cite[Equations (7.11) and (7.12)]{Stanley2001}
that 
\[
H(z)=\prod_{n=1}^{\infty}(1-x_{n}z)^{-1}\quad\text{and}\quad E(z)=\prod_{n=1}^{\infty}(1+x_{n}z),
\]
from which the identity
\begin{equation}
H(z)=E(-z)^{-1}\label{e-he}
\end{equation}
follows. It is then a consequence of \cite[Proposition 7.7.4]{Stanley2001}
that
\begin{equation}
H(z)=\exp\Big(\sum_{k=1}^{\infty}\frac{p_{k}}{k}z^{k}\Big)\label{e-hp}
\end{equation}
and 
\begin{equation}
E(z)=\exp\Big(\sum_{k=1}^{\infty}(-1)^{k-1}\frac{p_{k}}{k}z^{k}\Big).\label{e-ep}
\end{equation}
We adopt the notation $H\coloneqq H(1)=\sum_{n=0}^{\infty}h_{n}$
and $E\coloneqq E(1)=\sum_{n=0}^{\infty}e_{n}$.

For a partition $\lambda=(1^{m_{1}}2^{m_{2}}\cdots)$, let
$z_{\lambda}\coloneqq1^{m_{1}}m_{1}!\,2^{m_{2}}m_{2}!\cdots$.
An important property of the numbers $z_\lambda$ is that if $\lambda$ is a partition of $n$ then $n!/z_\lambda$ is the number of permutations in $\mathfrak{S}_n$ of cycle type $\lambda$, i.e., permutations in which the lengths of the cycles correspond to the parts of $\lambda$.

The following lemma is helpful in working with expressions like \eqref{e-hp}. We omit the proof, which is a straightforward computation (cf.~Macdonald \cite[p.~25]{Macdonald1995}). 

\begin{lem}
\label{l-expsum}
For any sequence $a_1,a_2,\dots$ we have
\begin{equation*}
\exp\Big(\sum_{k=1}^{\infty}\frac{a_{k}}{k}x^k \Big)=\sum_{\lambda}\frac{x^{\left|\lambda\right|}}{z_\lambda}\prod_{k=1}^{l(\lambda)}a_{\lambda_k},
\end{equation*}
where the sum on the right is over all partitions $\lambda$.
\end{lem}

It follows from \eqref{e-hp} and Lemma \ref{l-expsum} that 
$H=\sum_{\lambda}p_{\lambda}/z_{\lambda}$.

Let $\left\langle \cdot\, ,\cdot\right\rangle \colon\Lambda\times\Lambda\rightarrow\mathbb{Q}$
denote the usual scalar product on symmetric functions defined by 
\[
\left\langle m_{\lambda},h_{\tau}\right\rangle =\begin{cases}
1, & \text{if }\lambda=\tau,\\
0, & \text{otherwise,}
\end{cases}
\]
for all partitions $\lambda$ and $\tau$ and extending bilinearly,
that is, by requiring that $\{m_{\lambda}\}$ and $\{h_{\tau}\}$
be dual bases. Then we have
\[
\left\langle p_{\lambda},p_{\tau}\right\rangle =\begin{cases}
z_{\lambda}, & \text{if }\lambda=\tau,\\
0, & \text{otherwise.}
\end{cases}
\]
for all $\lambda$ and $\tau$ \cite[Proposition 7.9.3]{Stanley2011}.

We extend the scalar product in the obvious way to symmetric functions of unbounded degree and symmetric functions with coefficients that involve other variables such as $t$,  $z$, and~$\alpha$. Note that the scalar product is not always defined for unbounded symmetric functions: for example, 
$\left\langle H(z), H\right\rangle=\sum_{n=0}^\infty z^n = (1-z)^{-1}$ but $\left\langle H,H\right\rangle$ is undefined.

\subsection{Plethysm}

The aim of this section is to give a brief, self-contained introduction
to plethysm and to develop a few preliminary lemmas involving plethystic
substitutions that we will need later in this paper. None of the results
in this section are new, but we have chosen to provide proofs in order to help guide readers who are unfamiliar with plethysm.

Let $R$ be a commutative ring containing $\mathbb{Q}$. We define
a \textit{lambda ring} over $R$ to be a commutative $R$-algebra
$A$ together with an $R$-algebra endomorphism $\psi_{i}\colon A\rightarrow A$
for every $i\in\mathbb{P}$ (the set of positive integers) such that
$\psi_{1}$ is the identity map, and $\psi_{i}\circ\psi_{j}=\psi_{ij}$
for every $i,j\in\mathbb{P}$. These homomorphisms $\psi_{i}$ are
called \textit{Adams operations}. We define an operation $\Lambda\times A\rightarrow A$,
where the image of $(f,a)\in\Lambda\times A$ is denoted $f[a]$,
by these two properties:
\begin{enumerate}
\item For any $i\geq1$, $p_{i}[a]=\psi_{i}(a)$.
\item For any fixed $a\in A$, the map $f\mapsto f[a]$ is an $R$-algebra
homomorphism from $\Lambda$ to $A$.
\end{enumerate}
Throughout this paper, we will take our ring $R$ to be $\mathbb{Q}$
and our lambda ring $A$ to be a $\mathbb{Q}$-algebra of formal power
series containing $\Lambda$ as a subalgebra. Moreover, from this
point on we shall take the $i$th Adams operation $\psi_{i}$
to be the result of replacing each variable with its $i$th power.
For a symmetric function $f\in\Lambda$, this means that $p_{i}[f]=\psi_{i}(f(x_{1},x_{2},\dots))=f(x_{1}^{i},x_{2}^{i},\dots)$
but more generally, if $f$ contains other variables, then in $p_i[f]$ they are all 
raised to the $i$th power as well. For example, $p_{i}[qt^{2}p_{m}]=\psi_{i}(qt^{2}p_{m})=q^{i}t^{2i}p_{im}$.
The map $(f,a)\mapsto f[a]$ is called \textit{plethysm}. The terms ``composition'' and ``plethystic substitution'' are also used for this operation. (Sometimes the term ``plethysm'' is restricted to the case in which $A=\Lambda$, but we will use it for the more general operation.)

As with the scalar product, we extend plethysm in the obvious way to symmetric functions of unbounded degree with coefficients involving other variables. If $f$ is a symmetric function of unbounded degree and if $a$ has a nonzero constant term, then $f[a]$ need not be defined. For example, $H(z)[1]=(1-z)^{-1}$ but $H[1]$ is undefined. In some of our formulas (e.g., Lemma \ref{l-Hps}), we assume implicitly that any infinite sums involved converge as formal power series. 

Note that plethysm does not commute with evaluation of variables: if $\alpha$ is a variable then $p_n[\alpha] = \alpha^n$, but if $y$ is a rational number then $p_n[y]=y$, so 
$p_n[\alpha]\rvert_{\alpha=y}\ne p_n[y]$.

We call $\mathsf{m}\in A$ a \emph{monic term} if it is a monomial with coefficient 1.
The next theorem gives a straightforward method for evaluating the
plethystic substitution $f[a]$ if $a$ is expressed as a sum of monic terms. 
\begin{thm}
\label{t-monic} Suppose that $a\in A$ can be expressed as a \textup{(}possibly
infinite\textup{)} sum of monic terms $\mathsf{m}_{1}+\mathsf{m}_{2}+\cdots$.
Then for any $f=f(x_{1},x_{2},\dots)\in\Lambda$, we have $f[a]=f(\mathsf{m}_{1},\mathsf{m}_{2},\dots)$.
\end{thm}

\begin{proof}
It is sufficient to prove this result for $f=p_{i}$, and this is
straightforward: 
\begin{align*}
p_{i}[a] & =\psi_{i}(\mathsf{m}_{1}+\mathsf{m}_{2}+\cdots)\\
 & =\mathsf{m}_{1}^{i}+\mathsf{m}_{2}^{i}+\cdots\\
 & =p_{i}(\mathsf{m}_{1},\mathsf{m}_{2},\dots).\qedhere
\end{align*}
\end{proof}

For example, 1 is a monic term, so if $k$ is a positive integer and $f\in \Lambda$ then
\begin{equation*}
f[k]=f(\underbrace{1,1,\dots,1}_k).
\end{equation*}

In what follows, we let $X\coloneqq x_{1}+x_{2}+\cdots=p_{1}$, so
that $f[X]=f$. The next several lemmas concern plethystic substitutions
of the form $H[a]$ for certain kinds of elements $a\in A$. 

\begin{lem}
\label{l-Hps} Let $f,g\in A$, let $\mathsf{m}\in A$ be a monic term, and let $k\in\mathbb{Z}$. Then
\begin{enumerate}
\item [\normalfont{(a)}] $H[f+g]=H[f]H[g]$,
\item [\normalfont{(b)}] $H[\mathsf{m}f]=H(\mathsf{m})[f]$, and
\item [\normalfont{(c)}]  $H[kf]=H[f]^{k}$.
\end{enumerate}
\end{lem}

\begin{proof}
{\allowdisplaybreaks
First, we have
\begin{align*}
H[f+g] & =\exp\Big(\sum_{n=1}^{\infty}\frac{p_{n}}{n}\Big)[f+g]\\
 & =\exp\Big(\sum_{n=1}^{\infty}\frac{p_{n}}{n}[f+g]\Big)\\
 & =\exp\Big(\sum_{n=1}^{\infty}\frac{1}{n}(p_{n}[f]+p_{n}[g])\Big)\\
 & =\exp\Big(\sum_{n=1}^{\infty}\frac{p_{n}}{n}\Big)[f]\cdot\exp\Big(\sum_{n=1}^{\infty}\frac{p_{n}}{n}\Big)[g]\\
 & =H[f]H[g]
\end{align*}
} which proves (a). For (b), we have
\begin{equation*}
H[\mathsf{m}f] = \exp\Big(\sum_{n=1}^\infty \frac{p_n}{n}\Big)[\mathsf{m}f]
  =\exp\Big(\sum_{n=1}^\infty \mathsf{m}^n \frac{p_n[f]}{n}\Big)
  =\exp\Big(\sum_{n=1}^\infty \mathsf{m}^n \frac{p_n}{n}\Big)[f]
  =H(\mathsf{m})[f].
\end{equation*}
We omit the proof of (c), which is similar to that of (a).
\end{proof}
\begin{lem}
\label{l-HE} Let $\mathsf{m}\in A$ be a monic term, let $\alpha\in A$
be a variable, and let $k\in\mathbb{Z}$. Then
\end{lem}

\begin{enumerate}
\item [\normalfont{(a)}] $H[\mathsf{m}X]=H(\mathsf{m})$,
\item [\normalfont{(b)}] $H[-\mathsf{m}X]=E(-\mathsf{m})$,
\item [\normalfont{(c)}] $H[\mathsf{m}]=1/(1-\mathsf{m})$, \textit{and}
\item [\normalfont{(d)}]  $H(\mathsf{m})[k(1-\alpha)]=(1-\mathsf{m}\alpha)^{k}/(1-\mathsf{m})^{k}$.
\end{enumerate}
\begin{proof}
Part (a) is a special case of Lemma \ref{l-Hps} (b), and we have
\begin{align*}
H[-\mathsf{m}X] & =H[\mathsf{m}X]^{-1} && \text{(by Lemma }\ref{l-Hps}\text{ (c))}\\
 & =H(\mathsf{m})^{-1} && \text{(by Lemma }\ref{l-HE}\text{ (a))}\\
 & =E(-\mathsf{m}) && \text{(by (}\ref{e-he}\text{))}
\end{align*}
which proves part (b). Next, by Theorem \ref{t-monic} we have 
\[
H[\mathsf{m}]=\sum_{n=0}^{\infty}h_{n}(\mathsf{m},0,0,\dots)=\sum_{n=0}^{\infty}\mathsf{m}^{n}=\frac{1}{1-\mathsf{m}};
\]
this proves part (c). Finally, to prove part (d), observe that by Lemma \ref{l-Hps} (c) it suffices to prove the case $k=1$. We have 
\begin{align*}
H(\mathsf{m})[1-\alpha]
  &=H[\mathsf{m}(1-\alpha)] &&\text{(by Lemma \ref{l-Hps} (b))}\\
  &=H[\mathsf{m}]H[-\mathsf{m}\alpha] &&\text{(by Lemma \ref{l-Hps} (a))}\\
  &=\frac{H[\mathsf{m}]}{H[\mathsf{m}\alpha]}  &&\text{(by Lemma \ref{l-Hps} (c))}\\
  &=\frac{1-\mathsf{m}\alpha}{1-\mathsf{m}}&&\text {(by Lemma \ref{l-HE} (c))},
\end{align*}
thus completing the proof.
\end{proof}
\begin{lem}
\label{l-plethHsp} Suppose that $g\in A$ does not contain any of the variables $x_{1},x_{2},\dots$. Then for any symmetric function $f$, we have $f[g]=\left\langle f[X],H[gX]\right\rangle $.
\end{lem}

\begin{proof}
We prove this result for $f=p_{\lambda}$, and then the result follows
by linearity. First, it is easy to show that for any  partition
$\tau$, we have $p_{\tau}[gX]=p_{\tau}[g]p_{\tau}[X]$. Thus

\begin{align*}
H[gX] & =\sum_{\tau}\frac{1}{z_{\tau}}p_{\tau}[gX]=\sum_{\tau}\frac{1}{z_{\tau}}p_{\tau}[g]p_{\tau}[X]
\end{align*}
and so
\begin{align*}
\left\langle p_{\lambda}[X],H[gX]\right\rangle  & =\biggl\langle  p_{\lambda},\sum_{\tau}\frac{1}{z_{\tau}}p_{\tau}[g]p_{\tau}\biggr\rangle .
\end{align*}
Since $g$ does not contain any of the variables $x_{1},x_{2},\dots$,
the same is true for $p_{\tau}[g]$ for any $\tau$. Thus we can pull
out $p_{\tau}[g]$ from the scalar product expression to obtain 
\[
\left\langle p_{\lambda}[X],H[gX]\right\rangle =\sum_{\tau}\left\langle p_{\lambda},\frac{p_{\tau}}{z_{\tau}}\right\rangle p_{\tau}[g]=p_{\lambda}[g]
\]
and we are done.
\end{proof}

\subsection{\label{ss-qsymgf}Descent compositions, cycle type, and quasisymmetric
generating functions}

We use the notation $L\vDash n$ to indicate that $L$ is a composition
of $n$. Every permutation can be uniquely decomposed into a sequence
of maximal increasing consecutive subsequences\textemdash or equivalently,
maximal consecutive subsequences containing no descents\textemdash which
we call \textit{increasing runs}. The \textit{descent composition}
of $\pi$, denoted $\Comp(\pi)$, is the composition whose parts are
the lengths of the increasing runs of $\pi$ in the order that they
appear. For example, the increasing runs of $\pi=85712643$ are $8$,
$57$, $126$, $4$, and $3$, so the descent composition of $\pi$
is $\Comp(\pi)=(1,2,3,1,1)$. 

For a composition $L=(L_{1},L_{2},\dots,L_{k})$, let $\Des(L)\coloneqq\{L_{1},L_{1}+L_{2},\dots,L_{1}+\cdots+L_{k-1}\}$.
It is easy to see that if $L$ is the descent composition of $\pi$,
then $\Des(L)$ is the set of descents (i.e., \textit{descent set})
of $\pi$. Recall that the fundamental quasisymmetric function $F_{L}$
is defined by
\[
F_{L}\coloneqq\sum_{\substack{i_{1}\leq i_{2}\leq\cdots\leq i_{n}\\
i_{j}<i_{j+1}\,\mathrm{if}\,j\in\Des(L)
}
}x_{i_{1}}x_{i_{2}}\cdots x_{i_{n}}.
\]
Given a set $\Pi$ of permutations, its \textit{quasisymmetric generating
function} $Q(\Pi)$ is defined by 
\[
Q(\Pi)\coloneqq\sum_{\pi\in\Pi}F_{\Comp(\pi)}.
\]

Moreover, given a composition $L$, let $r_{L}$ denote the skew Schur
function of ribbon shape $L$.\footnote{Ribbon shapes are also  called ``skew-hooks'' (e.g., by Gessel\textendash Reutenauer
\cite{Gessel1993}) or ``border strips'' (e.g., by Macdonald \cite{Macdonald1995} and Stanley \cite{Stanley2007,Stanley2001}).} 
Thus $r_L$ is defined by 
\[
{r}_{L}=\sum_{i_1,\dots, i_n}x_{i_{1}}x_{i_{2}}\cdots x_{i_{n}}
\]
where the sum is over all $i_{1},\dots,i_{n}$  satisfying
\begin{equation*}
\underset{L_{1}}{\underbrace{i_{1}\leq\cdots\leq i_{L_{1}}}}>\underset{L_{2}}{\underbrace{i_{L_{1}+1}\leq\cdots\leq i_{L_{1}+L_{2}}}}>\cdots>\underset{L_{k}}{\underbrace{i_{L_{1}+\cdots+L_{k-1}+1}\leq\cdots\leq i_{n}}}.
\end{equation*}

The following is \cite[Corollary 4]{Gessel1984}.
\begin{thm}
\label{t-QrL} Suppose that $Q(\Pi)$ is a symmetric function. Then the number of permutations in $\Pi$ with descent composition $L$ is equal to
$\left\langle Q(\Pi),r_{L}\right\rangle $.
\end{thm}

Recall that a permutation $\pi$ has \textit{cycle type} $\lambda=(1^{m_{1}}2^{m_{2}}\cdots)$
if $\pi$ has exactly $m_{1}$ cycles of length 1, $m_{2}$ cycles of
length 2, and so on. Henceforth, cycles of length $i$ are called $i$-\textit{cycles},
1-cycles in particular are called \textit{fixed points}, and the number
of fixed points of a permutation $\pi$ is denoted $\fix(\pi)$. For
$n\in\mathbb{P}$, define the symmetric function $L_{n}$ by 
\[
L_{n}\coloneqq\frac{1}{n}\sum_{d\mid n}\mu(d)p_{d}^{n/d}
\]
where $\mu$ is the number-theoretic M{\"o}bius function. Then, given
a partition $\lambda=(1^{m_{1}}2^{m_{2}}\cdots)$, define $L_{\lambda}$
by 
\[
L_{\lambda}\coloneqq h_{m_{1}}[L_{1}]h_{m_{2}}[L_{2}]\cdots.
\]
The symmetric functions $L_{\lambda}$ are called \textit{Lyndon symmetric
functions}. 
Gessel and Reutenauer \cite[Theorem 2.1]{Gessel1993} showed
that $L_{\lambda}$ is the quasisymmetric generating function for
the set of permutations with cycle type $\lambda$.
\begin{cor}
The number of permutations $\pi$ with cycle type $\lambda$ and descent
composition $M$ is equal to $\left\langle L_{\lambda},r_{M}\right\rangle $.
\end{cor}

\subsection{Ribbon expansions}

In the work that will follow, we will need to expand the symmetric
function expressions $(1-tE(yx)H(x))^{-1}$, $H(x)/(1-tE(yx)H(x))$,
and $(1+tH(x))/(1-t^{2}E(x)H(x))$ in terms of the ribbon Schur functions
$r_{L}$ and in terms of the power sums $p_{\lambda}$. First we give
the ribbon expansions, which reveal connections between these expressions
and permutation statistics. 

We let $\des(L)$, $\pk(L)$, $\lpk(L)$, and $\udr(L)$ be equal
to the values $\des(\pi)$, $\pk(\pi)$, $\lpk(\pi)$, and $\udr(\pi)$,
respectively, of any permutation $\pi$ with descent composition $L$.
These are well-defined because the statistics $\des$, $\pk$, $\lpk$,
and $\udr$ depend only on the descent composition; in other words,
$\Comp(\pi)=\Comp(\sigma)$ implies $\st(\pi)=\st(\sigma)$ for $\st=\des,\pk,\lpk,\udr$. Recall that permutation statistics with this property are called \textit{descent statistics}. In general, if $\st$ is a descent statistic, then we let $\st(L)$ denote the value of $\st$ on any permutation $\pi$ with descent composition $L$.

\begin{lem}
\label{l-ribexp} We have the formulas
\leqnomode
\begin{multline}
\tag{a}\frac{1}{1-tE(yx)H(x)}=\frac{1}{1-t}\\
+\frac{1}{1+y}\sum_{n=1}^{\infty}\sum_{L\vDash n}\left(\frac{1+yt}{1-t}\right)^{n+1}\left(\frac{(1+y)^{2}t}{(y+t)(1+yt)}\right)^{\pk(L)+1}\left(\frac{y+t}{1+yt}\right)^{\des(L)+1}x^{n}r_{L},
\end{multline}
\vspace{-9pt}
\begin{multline}
\tag{b}
\quad
\frac{H(x)}{1-tE(yx)H(x)}=\frac{1}{1-t}\\
+\sum_{n=1}^{\infty}\sum_{L\vDash n}\frac{(1+yt)^{n}}{(1-t)^{n+1}}\left(\frac{(1+y)^{2}t}{(y+t)(1+yt)}\right)^{\lpk(L)}\left(\frac{y+t}{1+yt}\right)^{\des(L)}x^{n}r_{L},\qquad
\end{multline}
\vspace{-9pt}and
\begin{equation}
\tag{c}
\frac{1+tH(x)}{1-t^{2}E(x)H(x)}=\frac{1}{1-t}+\frac{1}{2(1-t)^{2}}\sum_{n=1}^{\infty}\sum_{L\vDash n}\frac{(1+t^{2})^{n}}{(1-t^{2})^{n-1}}\left(\frac{2t}{1+t^{2}}\right)^{\udr(L)}x^{n}r_{L}.
\end{equation}
\end{lem}

\begin{proof}
These formulas are obtained directly from Lemma 4.1, Lemma 4.6, and
Corollary 4.12 of \cite{Zhuang2017}\textemdash which are noncommutative
versions of these formulas\textemdash via the canonical projection
map from noncommutative symmetric functions
to symmetric functions. 
\end{proof}

We note that part (a) of Lemma \ref{l-ribexp} can be rewritten as 
\begin{align*}
\frac{1}{1-vE(ux)H(x)} & =\frac{1}{1-v}+\frac{1}{1+u}\sum_{n=1}^{\infty}\sum_{L\vDash n}\left(\frac{1+uv}{1-v}\right)^{n+1}y^{\pk(L)+1}t^{\des(L)+1}x^{n}r_{L}
\end{align*}
where 
\[
u=\frac{1+t^{2}-2yt-(1-t)\sqrt{(1+t)^{2}-4yt}}{2(1-y)t}
\]
and 
\[
v=\frac{(1+t)^{2}-2yt-(1+t)\sqrt{(1+t)^{2}-4yt}}{2yt}.
\]
To see this, we replace $y$ with $u$ and replace $t$ with $v$ in
part (a). Then we set $y=\frac{(1+u)^{2}v}{(u+v)(1+uv)}$ and $t=\frac{u+v}{1+uv}$;
solving these two equations yields the above expressions for $u$
and $v$. By the same reasoning, part (b) of Lemma \ref{l-ribexp} is equivalent to 
\begin{align*}
\frac{H(x)}{1-vE(ux)H(x)} & =\frac{1}{1-v}+\sum_{n=1}^{\infty}\sum_{L\vDash n}\frac{(1+uv)^{n}}{(1-v)^{n+1}}y^{\lpk(L)}t^{\des(L)}x^{n}r_{L}
\end{align*}
where $u$ and $v$ are the same as above, and part (c) can be ``inverted'' in a similar way.

There are explicit formulas for expressing powers of $u$ and $v$ in terms of $y$ and $t$.
It is easy to check that $v=yt P^2$ and $u = (1-y)t Q^2$, where
\begin{equation*}
P =\frac{1+t-\sqrt{(1+t)^{2}-4yt}}{2yt}=\frac{2}{1+t+\sqrt{(1+t)^2 -4yt}}
\end{equation*}
and
\begin{equation*}
Q = -\frac{1-t-\sqrt{(1+t)^{2}-4yt}}{2(1-y)t} =\frac{2}{1-t+\sqrt{(1+t)^2 -4yt}}.
\end{equation*}
Using the fact that $P=1-tP+ytP^2$ and $Q=P/(1-tP)$, we can apply Lagrange inversion  
(see, e.g., \cite[Equation (2.4.5)]{Gessel2016a}) to find an explicit formula for the expansion of $P^aQ^b$ in powers of $y$ and $t$,
\begin{equation*}
P^a Q^b = \sum_{i,j=0}^\infty(-1)^{j-i}\left[\binom{a+b+i+j-1}{i}\binom{a+j-1}{j-i} - \binom{a+b+i+j-1}{i-1}\binom{a+j}{j-i}\right]y^i t^j,
\end{equation*}
which yields explicit formulas for expanding powers of $u$ and $v$.

Furthermore, it is interesting to note that part (c) of Lemma \ref{l-ribexp} can be derived from parts (a) and (b).  
First, we can split the left side of (c) into even and odd powers of $t$.
Using the fact that the valley number\footnote{Given $\pi\in\mathfrak{S}_{n}$, we say that $i\in\{2,\dots,n\}$
is a \textit{valley} of $\pi$ if $\pi(i-1)>\pi(i)<\pi(i+1)$. Then $\val(\pi)$ is the number of valleys of $\pi$.}, $\val$, has the same distribution as $\pk$, 
we may apply the case $y=1$ of (a) and (b) to express these two sums in terms of $\val$ and $\lpk$. Then we may use the identities
 $\val(L) +1 = \left\lceil{\udr(L)/2}\right\rceil$ and $\lpk(L) =\left\lfloor\udr(L)/2\right\rfloor$ (see \cite[Lemma 2.2]{Gessel2018}) to express these two sums in terms of $\udr$. Finally we split these two sums into even and odd values of $\udr(L)$ and rearrange to obtain the right side of (c).

Before continuing, we present here symmetric function formulas analogous to those in Lemma \ref{l-ribexp} for two more descent statistics: the \emph{double descent number} $\ddes$ and the \emph{number of biruns} $\br$. We will not discuss them further, but they could be used to derive formulas analogous to those we will give later on for $\des$, $\pk$, $\lpk$, and $\udr$.

A \emph{double descent} of a permutation $\pi\in \mathfrak{S}_n$ is an index $i$ with $2\le i \le n-1$ such that $\pi(i-1)>\pi(i)>\pi(i+1)$. 
We denote by $\ddes(\pi)$ the number of double descents of $\pi$. (See, for example, \cite[Theorem 12]{Zhuang2016} for the generating function for permutations by double descents.) Using the fact that each descent of a permutation $\pi$ either occurs in position 1 or is preceded by a descent or an ascent, we see that $\ddes(\pi) + \lpk(\pi) = \des(\pi)$, so $\ddes(\pi) = \des(\pi) - \lpk(\pi)$.  
Thus, we might expect that a symmetric function formula for double descents can be derived from Lemma \ref{l-ribexp} (b). To obtain the simplest formula for double descents, we first replace $t$ with $t^2$ in Lemma \ref{l-ribexp} (b), then set $y=1/t$, and then replace $x$ with $tx$. This gives 
\begin{equation}
\label{e-dd}
\frac{H(tx)}{1-t^2 E(x) H(tx)} = \frac{1}{1-t^2} +\frac{1}{1-t^2}
  \sum_{n=1}^\infty \sum_{L\vDash n} \left(\frac{t}{1-t}\right)^{n}
  (t-1+t^{-1})^{\ddes(L)}x^n r_L.
\end{equation}

Recall that a birun of a permutation is a maximal monotone consecutive subsequence. We denote by $\br(\pi)$ the number of biruns of $\pi$. Since biruns are closely related to up-down runs, we might expect a formula similar to Lemma \ref{l-ribexp} (c) for biruns, and in fact there is such a formula:
\begin{multline}
\label{e-biruns}
\qquad
\frac{2+tH(x)+tE(x)}{1-t^2 E(x)H(x)}=\frac{2}{1-t}+\frac{2t}{(1-t)^2}x h_1\\
  + \frac{(1+t)^3}{2(1-t)}
\sum_{n=2}^{\infty}\sum_{L\vDash n}\frac{(1+t^{2})^{n-1}}{(1-t^{2})^n}\left(\frac{2t}{1+t^{2}}\right)^{\br(L)}x^{n}r_{L}.
\qquad
\end{multline}
Formula \eqref{e-biruns} can be proved using the approach of \cite{Zhuang2016}.

\subsection{Power sum expansions and Eulerian polynomials}
Next, we give the power sum expansions of $(1-tE(yx)H(x))^{-1}$, $H(x)/(1-tE(yx)H(x))$,
and $(1+tH(x))/(1-t^{2}E(x)H(x))$, which\textemdash perhaps
surprisingly\textemdash involve Eulerian polynomials and type B Eulerian
polynomials. The $n$th \textit{type B Eulerian polynomial} $B_{n}(t)$
may be defined by the formula
\[
\frac{B_{n}(t)}{(1-t)^{n+1}}=\sum_{k=0}^{\infty}(2k+1)^{n}t^{k}
\]
and gives the distribution of the type B descent number over the $n$th
hyperoctahedral group; see \cite[Section 2.3]{Zhuang2017} for details.
Recall that $l(\lambda)$ is the number of parts of the partition
$\lambda$. We define $o(\lambda)$ to be the number of odd
parts of $\lambda$ and we write $\sum_{\lambda\;\mathrm{odd}}$ to denote a sum over partitions $\lambda$ in which every part is odd.
\begin{lem}
\label{l-psexp} We have the formulas
\leqnomode
\[\tag{a}
\frac{1}{1-tE(yx)H(x)}=\sum_{\lambda}\frac{p_{\lambda}}{z_{\lambda}}\frac{A_{l(\lambda)}(t)}{(1-t)^{l(\lambda)+1}}x^{\left|\lambda\right|}\prod_{k=1}^{l(\lambda)}(1-(-y)^{\lambda_{k}}),
\]
\[\tag{b}
\frac{H(x)}{1-tE(x)H(x)}=\sum_{\lambda}\frac{p_{\lambda}}{z_{\lambda}}\frac{B_{o(\lambda)}(t)}{(1-t)^{o(\lambda)+1}}x^{\left|\lambda\right|},
\]
and
\begin{align*}
\tag{c}
\frac{1+tH(x)}{1-t^{2}E(x)H(x)} & =\sum_{\lambda\;\mathrm{odd}}\frac{p_{\lambda}}{z_{\lambda}}2^{l(\lambda)}\frac{A_{l(\lambda)}(t^{2})}{(1-t^{2})^{l(\lambda)+1}}x^{\left|\lambda\right|}+t\sum_{\lambda}\frac{p_{\lambda}}{z_{\lambda}}\frac{B_{o(\lambda)}(t^{2})}{(1-t^{2})^{o(\lambda)+1}}x^{\left|\lambda\right|}.
\end{align*}
\end{lem}

\begin{proof}
First, we have
\begin{align*}
\frac{1}{1-tE(yx)H(x)} & =\sum_{n=0}^{\infty}t^{n}\exp\Big(n\sum_{k=1}^{\infty}\frac{p_{k}}{k}(-1)^{k-1}y^{k}x^{k}\Big)\exp\Big(n\sum_{k=1}^{\infty}\frac{p_{k}}{k}x^{k}\Big)\\
 & =\sum_{n=0}^{\infty}t^{n}\exp\Big(n\sum_{k=1}^{\infty}\frac{p_{k}}{k}(1-(-y)^{k})x^{k}\Big).
\end{align*}
By Lemma \ref{l-expsum}, we have
\[
\exp\Big(n\sum_{k=1}^{\infty}\frac{p_{k}}{k}(1-(-y)^{k})x^{k}\Big)=\sum_{\lambda}\frac{p_{\lambda}}{z_{\lambda}}n^{l(\lambda)}x^{\left|\lambda\right|}\prod_{k=1}^{l(\lambda)}(1-(-y)^{\lambda_{k}}).
\]
It follows that
\begin{align*}
\frac{1}{1-tE(yx)H(x)} & =\sum_{n=0}^{\infty}t^{n}\sum_{\lambda}\frac{p_{\lambda}}{z_{\lambda}}n^{l(\lambda)}x^{\left|\lambda\right|}\prod_{k=1}^{l(\lambda)}(1-(-y)^{\lambda_{k}})\\
 & =\sum_{\lambda}\frac{p_{\lambda}}{z_{\lambda}}\sum_{n=0}^{\infty}n^{l(\lambda)}t^{n}x^{\left|\lambda\right|}\prod_{k=1}^{l(\lambda)}(1-(-y)^{\lambda_{k}})\\
 & =\sum_{\lambda}\frac{p_{\lambda}}{z_{\lambda}}\frac{A_{l(\lambda)}(t)}{(1-t)^{l(\lambda)+1}}x^{\left|\lambda\right|}\prod_{k=1}^{l(\lambda)}(1-(-y)^{\lambda_{k}}),
\end{align*}
where in the last step we are using (\ref{e-eulerian}). This proves
part (a). Note that setting $y=1$ in (a) gives
\begin{align}
\label{e-p-peaks}
\frac{1}{1-tE(x)H(x)}=\sum_{\lambda \textup{ odd}}\frac{p_{\lambda}}{z_{\lambda}}2^{l(\lambda)}\frac{A_{l(\lambda)}(t)}{(1-t)^{l(\lambda)+1}}x^{\left|\lambda\right|}.
\end{align}

To prove part (b), we begin with
\begin{align*}
\frac{H(x)}{1-tE(yx)H(x)} & =\exp\Big(\sum_{k=1}^{\infty}\frac{p_{k}}{k}x^{k}\Big)\sum_{n=0}^{\infty}t^{n}\exp\Big(n\sum_{k=1}^{\infty}\frac{p_{k}}{k}(-1)^{k-1}y^{k}x^{k}\Big)\exp\Big(n\sum_{k=1}^{\infty}\frac{p_{k}}{k}x^{k}\Big)\\
 & =\sum_{n=0}^{\infty}t^{n}\exp\Big(\sum_{k=1}^{\infty}\frac{p_{k}}{k}(n(1-(-y)^{k})+1)x^{k}\Big)\\
 & =\sum_{n=0}^{\infty}t^{n}\sum_{\lambda}\frac{p_{\lambda}}{z_{\lambda}}x^{\left|\lambda\right|}\prod_{k=1}^{l(\lambda)}(n(1-(-y)^{\lambda_{k}})+1).
\end{align*}
Setting $y=1$, we obtain 
\begin{align*}
\frac{H(x)}{1-tE(x)H(x)} & =\sum_{n=0}^{\infty}t^{n}\sum_{\lambda}\frac{p_{\lambda}}{z_{\lambda}}x^{\left|\lambda\right|}\prod_{k=1}^{l(\lambda)}(n(1-(-1)^{\lambda_{k}})+1)\\
 & =\sum_{n=0}^{\infty}t^{n}\sum_{\lambda}\frac{p_{\lambda}}{z_{\lambda}}x^{\left|\lambda\right|}(2n+1)^{o(\lambda)}\\
 & =\sum_{\lambda}\frac{p_{\lambda}}{z_{\lambda}}\frac{B_{o(\lambda)}(t)}{(1-t)^{o(\lambda)+1}}x^{\left|\lambda\right|}.
\end{align*}

Finally, using parts (a) and (b), we have
\begin{align*}
\frac{1+tH(x)}{1-t^{2}E(x)H(x)} & =\frac{1}{1-t^{2}E(x)H(x)}+t\frac{H(x)}{1-t^{2}E(x)H(x)}\\
 & =\sum_{\lambda\text{ odd}}\frac{p_{\lambda}}{z_{\lambda}}2^{l(\lambda)}\frac{A_{l(\lambda)}(t^{2})}{(1-t^{2})^{l(\lambda)+1}}x^{\left|\lambda\right|}+t\sum_{\lambda}\frac{p_{\lambda}}{z_{\lambda}}\frac{B_{o(\lambda)}(t^{2})}{(1-t^{2})^{o(\lambda)+1}}x^{\left|\lambda\right|},
\end{align*}
thus proving part (c).
\end{proof}

\section{General plethystic formulas}

In this section, we derive general formulas\textemdash analogous to (\ref{e-desmaj})\textemdash that
will allow us to compute the joint distribution of $\pk$ and $\des$,
the joint distribution of $\lpk$ and $\des$, and the distribution
of $\udr$ over any set $\Pi\subseteq\mathfrak{S}_{n}$ of permutations
whose quasisymmetric generating function $Q(\Pi)$ is symmetric. Our formulas will involve the following map: Given $y\in A$ and an integer $k\in\mathbb{Z}$, define the homomorphism
$\Theta_{y,k}\colon\Lambda\rightarrow\mathbb{Q}[[y]]$ by
\[
\Theta_{y,k}(f)\coloneqq f[k(1-\alpha)]\rvert_{\alpha=-y}.
\]
That is, $\Theta_{y,k}$ first sends a symmetric function $f$ to the plethystic substitution $f[k(1-\alpha)]$, where $\alpha$ is a variable, and then evaluates this expression at $\alpha=-y$.

\subsection{General formula for peaks and descents}

Given a set $\Pi$ of permutations, define
\[
P^{(\pk,\des)}(\Pi;y,t)\coloneqq\sum_{\pi\in\Pi}y^{\pk(\pi)+1}t^{\des(\pi)+1},
\]
\[
P^{\pk}(\Pi;t)\coloneqq\sum_{\pi\in\Pi}t^{\pk(\pi)+1},\quad\text{and}\quad A(\Pi;t)\coloneqq\sum_{\pi\in\Pi}t^{\des(\pi)+1}.
\]
These encode the joint distribution of $\pk$ and $\des$, the distribution
of $\pk$, and the distribution of $\des$, respectively, over $\Pi$.
If $\Pi$ has a symmetric quasisymmetric function $Q(\Pi)$, then
the following theorem allows one to describe these polynomials in
terms of $\Theta_{y,k}(Q(\Pi))$. Moreover, if we know the power
sum expansion of $Q(\Pi)$, then this theorem allows us to describe
these polynomials in terms of Eulerian polynomials.

\begin{thm}
\label{t-pkdes} Let $\Pi\subseteq\mathfrak{S}_{n}$, with $n\ge1$, and suppose that
the quasisymmetric generating function $Q(\Pi)$ is a symmetric function
with power sum expansion $Q(\Pi)=q(p_1, p_2, p_3,\dots) = \sum_{\lambda\vdash n}c_{\lambda}p_{\lambda}$.
Then 
\leqnomode
\begin{multline*}
\tag{a}
\quad
\frac{1}{1+y}\left(\frac{1+yt}{1-t}\right)^{n+1}P^{(\pk,\des)}\left(\Pi;\frac{(1+y)^{2}t}{(y+t)(1+yt)},\frac{y+t}{1+yt}\right)=\sum_{k=0}^{\infty}\Theta_{y,k}(Q(\Pi))t^{k}\\
  =\sum_{\substack{\lambda\vdash n}}c_{\lambda}
    \frac{A_{l(\lambda)}(t)}{(1-t)^{l(\lambda)+1}}\prod_{k=1}^{l(\lambda)}(1-(-y)^{\lambda_{k}}),
\quad
\end{multline*}
\begin{multline*}
\tag{b}
\qquad
\frac{1}{2}\left(\frac{1+t}{1-t}\right)^{n+1}P^{\pk}\left(\Pi;\frac{4t}{(1+t)^{2}}\right)
  =\sum_{k=0}^{\infty}\Theta_{1,k}(Q(\Pi))t^{k}\\
  =\sum_{\substack{\lambda\vdash n\\
  \textrm{\textup{odd}}
}
}c_{\lambda}2^{l(\lambda)}\frac{A_{l(\lambda)}(t)}{(1-t)^{l(\lambda)+1}}
=\sum_{k=1}^n a_{k} 2^{k}\frac{A_{k}(t)}{(1-t)^{k+1}}
\qquad
\end{multline*}
where
\[\sum_{k=1}^n a_{k}w^k=q(w,0,w, 0, w, 0, \dots),\] 
and
\begin{equation*}
\tag{c}
\frac{A(\Pi;t)}{(1-t)^{n+1}}  =\sum_{k=0}^{\infty}\Theta_{0,k}(Q(\Pi))t^{k}
=\sum_{\lambda\vdash n}c_{\lambda}\frac{A_{l(\lambda)}(t)}{(1-t)^{l(\lambda)+1}}
=\sum_{k=1}^n b_{k}\frac{A_{k}(t)}{(1-t)^{k+1}}
\end{equation*}
where 
\[\sum_{k=1}^n b_{k} w^k=q(w,w,w,  \dots).\] 
\end{thm}

We note that taking $\Pi = \mathfrak{S}_n$ in Theorem \ref{t-pkdes} (a) and simplifying yields the formula
\begin{equation}
A_{n}(t)=\left(\frac{1+yt}{1+y}\right)^{n+1}P_{n}^{(\pk,\des)}\left(\frac{(1+y)^{2}t}{(y+t)(1+yt)},\frac{y+t}{1+yt}\right)\label{e-apkdes}
\end{equation}
due to the second author \cite[Theorem 4.2]{Zhuang2017}. Hence, Theorem \ref{t-pkdes} (a) generalizes (\ref{e-apkdes}) in the same way that (\ref{e-desmaj}) generalizes (\ref{e-qeulerian}).

The proof of Theorem \ref{t-pkdes} requires the following technical lemma.
\begin{lem}
\label{l-scalprodh} Let $f\in A$, let $\alpha\in A$ be a variable,
and let $k\in\mathbb{Z}$. Then $f[k(1-\alpha)]=\left\langle f,H^{k}E(-\alpha)^{k}\right\rangle $.
\end{lem}

\begin{proof}
Recall that $X = x_1 + x_2 + \cdots$. We have
\begin{align*}
f[k(1-\alpha)] & =\left\langle f[X],H[k(1-\alpha)X]\right\rangle && \text{(by Lemma }\ref{l-plethHsp}\text{)}\\
 & =\left\langle f[X],H[(1-\alpha)X]^{k}\right\rangle && \text{(by Lemma }\ref{l-Hps}\text{ (c))}\\
 & =\left\langle f[X],H[X]^{k}H[-\alpha X]^{k}\right\rangle  && \text{(by Lemma }\ref{l-Hps}\text{ (a))}\\
 & =\left\langle f,H^{k}E(-\alpha)^{k}\right\rangle  && \text{(by Lemma }\ref{l-HE}\text{ (b))},
\end{align*}
thus completing the proof.
\end{proof}

\begin{proof}[Proof of Theorem \ref{t-pkdes}]
We first prove part (a).
 Consider the following three expressions for the scalar product
$\left\langle Q(\Pi),(1-tE(y)H)^{-1}\right\rangle $. First, from
Lemma \ref{l-ribexp} (a) we have
\begin{multline*}
\left\langle Q(\Pi),\frac{1}{1-tE(y)H}\right\rangle \\
=\frac{1}{1+y}\sum_{L\vDash n}\left(\frac{1+yt}{1-t}\right)^{n+1}\left(\frac{(1+y)^{2}t}{(y+t)(1+yt)}\right)^{\pk(L)+1}\left(\frac{y+t}{1+yt}\right)^{\des(L)+1}\left\langle Q(\Pi),r_{L}\right\rangle 
\end{multline*}
which by Theorem \ref{t-QrL} simplifies to 
\begin{align*}
\left\langle Q(\Pi),\frac{1}{1-tE(y)H}\right\rangle  & =\frac{1}{1+y}\left(\frac{1+yt}{1-t}\right)^{n+1}\sum_{\pi\in\Pi}\left(\frac{(1+y)^{2}t}{(y+t)(1+yt)}\right)^{\pk(\pi)+1}\left(\frac{y+t}{1+yt}\right)^{\des(\pi)+1}\\
 & =\frac{1}{1+y}\left(\frac{1+yt}{1-t}\right)^{n+1}P^{(\pk,\des)}\left(\Pi;\frac{(1+y)^{2}t}{(y+t)(1+yt)},\frac{y+t}{1+yt}\right).
\end{align*}
Next, from Lemma \ref{l-scalprodh} we have
\begin{align*}
\left\langle Q(\Pi),\frac{1}{1-tE(y)H}\right\rangle  & =\sum_{k=0}^{\infty}\left\langle Q(\Pi),E(y)^{k}H^{k}\right\rangle t^{k}=\sum_{k=0}^{\infty}\Theta_{y,k}(Q(\Pi))t^{k}.
\end{align*}
{\allowdisplaybreaks
Finally, from Lemma \ref{l-psexp} (a) we have
\begin{align*}
\left\langle Q(\Pi),\frac{1}{1-tE(y)H}\right\rangle  & =\left\langle \sum_{\lambda\vdash n}c_{\lambda}p_{\lambda},\sum_{\tau}\frac{1}{z_{\tau}}\frac{A_{l(\tau)}(t)}{(1-t)^{l(\tau)+1}}p_{\tau}\prod_{k=1}^{l(\tau)}(1-(-y)^{\tau_{k}})\right\rangle \\
 & =\sum_{\substack{\lambda\vdash n\\
\tau
}
}\frac{c_{\lambda}}{z_{\tau}}\frac{A_{l(\tau)}(t)}{(1-t)^{l(\tau)+1}}\left\langle p_{\lambda},p_{\tau}\right\rangle \prod_{k=1}^{l(\tau)}(1-(-y)^{\tau_{k}})\\
 & =\sum_{\lambda\vdash n}c_{\lambda}\frac{A_{l(\lambda)}(t)}{(1-t)^{l(\lambda)+1}}\prod_{k=1}^{l(\lambda)}(1-(-y)^{\lambda_{k}}).
\end{align*}}
The desired result follows from equating these three expressions for
$\left\langle Q(\Pi),(1-tE(y)H)^{-1}\right\rangle $.

Setting $y=1$ in part (a) gives the first two equalities in part (b). The last equality in (b) comes from the fact that
$\sum_{\substack{\lambda\vdash n\\ \textrm{\textup{odd}}}} c_\lambda w^{l(\lambda)}$ is obtained from $Q(\Pi)$ by setting $p_i=w$ for $i$ odd and $p_i=0$ for $i$ even. Part (c) is obtained similarly. 
\end{proof}
We note that since $\Theta_{0,k}(Q(\Pi))=Q(\Pi)[k]$, the first equality of part (c) is the $q=1$ specialization of (\ref{e-desmaj}).

We may also refine Theorem \ref{t-pkdes} by additional permutation statistics. Given an integer-valued statistic $\st$, define
\[
Q^{\st}(\Pi)\coloneqq\sum_{\pi\in\Pi}F_{\Comp(\pi)}z^{\st(\pi)}
\]
and
\[
P^{(\pk,\des,\st)}(\Pi;y,t,z)\coloneqq\sum_{\pi\in\Pi}y^{\pk(\pi)+1}t^{\des(\pi)+1}z^{\st(\pi)}.
\]
Then the next theorem is proved in a similar way to Theorem \ref{t-pkdes} (a). We will also use analogous generalizations of parts (b) and (c) of Theorem \ref{t-pkdes}, but we will not state them here.

\begin{thm}
\label{t-pkdesst} Let $\Pi\subseteq\mathfrak{S}_{n}$, with $n\ge1$,  and let $\st$
be a permutation statistic. Suppose that $Q^{\st}(\Pi)$ is a symmetric
function \textup{(}in the variables $x_{1},x_{2},\dots$\textup{)}.
Then we have
\begin{align*}
\frac{1}{1+y}\left(\frac{1+yt}{1-t}\right)^{n+1}P^{(\pk,\des,\st)}\left(\Pi;\frac{(1+y)^{2}t}{(y+t)(1+yt)},\frac{y+t}{1+yt},z\right) & =\sum_{k=0}^{\infty}\Theta_{y,k}(Q^{\st}(\Pi))t^{k}.
\end{align*}
\end{thm}

\subsection{General formula for left peaks and descents}

We now prove an analogue of Theorem \ref{t-pkdes} for the joint distribution
of $\lpk$ and $\des$. Given a set of permutations $\Pi$, define
\[
P^{(\lpk,\des)}(\Pi;y,t)\coloneqq\sum_{\pi\in\Pi}y^{\lpk(\pi)}t^{\des(\pi)}\quad\text{and}\quad P^{\lpk}(\Pi;t)\coloneqq\sum_{\pi\in\Pi}t^{\lpk(\pi)}.
\]
\begin{thm}
\label{t-lpkdes} Let $\Pi\subseteq\mathfrak{S}_{n}$ and suppose
that the quasisymmetric generating function $Q(\Pi)$ is a symmetric
function with power sum 
expansion $Q(\Pi)=q(p_1, p_2, p_3,\dots) = \sum_{\lambda\vdash n}c_{\lambda}p_{\lambda}$.
Then 
\leqnomode
\begin{align*}
\tag{a}
\frac{(1+yt)^{n}}{(1-t)^{n+1}}P^{(\lpk,\des)}\left(\Pi;\frac{(1+y)^{2}t}{(y+t)(1+yt)},\frac{y+t}{1+yt}\right) & =\sum_{k=0}^{\infty}\Theta_{y,k}(Q(\Pi)[X+1])t^{k}
\end{align*}
and
\begin{multline*}
\tag{b}
\qquad
\frac{(1+t)^{n}}{(1-t)^{n+1}}P^{\lpk}\left(\Pi;\frac{4t}{(1+t)^{2}}\right)  =\sum_{k=0}^{\infty}\Theta_{1,k}(Q(\Pi)[X+1])t^{k}\\
  =\sum_{\lambda\vdash n}c_{\lambda}\frac{B_{o(\lambda)}(t)}{(1-t)^{o(\lambda)+1}}
  = \sum_{k=0}^n d_k \frac{B_k(t)}{(1-t)^{k+1}}
 \qquad
\end{multline*}
where
\begin{equation*}
\sum_{k=0}^n d_k w^k = q(w,1,w,1,w,1,\dots).
\end{equation*}
\end{thm}

We will need a lemma involving plethysm in order
to prove this theorem.

\begin{lem}
\label{l-spconst} Let $f,g\in A$, and let $\mathsf{m}\in A$ be a monic term. Then $\left\langle f[X+\mathsf{m}],g\right\rangle =\left\langle f,H[\mathsf{m}X]g\right\rangle $.
\end{lem}

In the proof of this lemma, we use the notation $\tau\cup\{n\}$ to denote the partition obtained from $\tau$ by adding a part of size $n$, and we write $\tau-\tau_{i}$ to denote the partition obtained by removing the part $\tau_{i}$ from $\tau$. 
\begin{proof}
By linearity, it suffices to prove the result for $f=m_{\lambda}$ and $g=h_{\tau}$.
First, because $X+\mathsf{m}$ is a sum of monic terms, we have
\begin{align*}
m_{\lambda}[X+\mathsf{m}] & =m_{\lambda}(\mathsf{m},x_{1},x_{2},\dots)\\
 & =m_{\lambda}+\sum_{\lambda_{i}}m_{\lambda-\lambda_{i}}\mathsf{m}^{\lambda_{i}}
\end{align*}
where the sum is over all distinct parts $\lambda_{i}$ of $\lambda$.
For example, if $\lambda=(2,2,1)$, then we have
\[
m_{(2,2,1)}[X+\mathsf{m}]=m_{(2,2,1)}+m_{(2,2)}\mathsf{m}+m_{(2,1)}\mathsf{m}^{2}.
\]
Thus,
\begin{align*}
\left\langle m_{\lambda}[X+\mathsf{m}],h_{\tau}\right\rangle  & =\sum_{\lambda_{i}}\left\langle m_{\lambda-\lambda_{i}},h_{\tau}\right\rangle \mathsf{m}^{\lambda_{i}}\\
 & =\begin{cases}
\mathsf{m}^{n}, & \text{if }\lambda=\tau\cup\{n\}\\
0, & \text{otherwise}
\end{cases}\\
 & =\sum_{n=0}^{\infty}\left\langle m_{\lambda},h_{\tau\cup\{n\}}\right\rangle \mathsf{m}^{n}\\
 & =\sum_{n=0}^{\infty}\left\langle m_{\lambda},h_{n}h_{\tau}\right\rangle \mathsf{m}^{n}\\
 & =\left\langle m_{\lambda},H(\mathsf{m})h_{\tau}\right\rangle \\
 & =\left\langle m_{\lambda},H[\mathsf{m}X]h_{\tau}\right\rangle 
\end{align*}
and we are done.
\end{proof}

\begin{proof}[Proof of Theorem \ref{t-lpkdes}]
 We first derive two expressions for $\left\langle Q(\Pi),H/(1-tE(y)H)\right\rangle $.
From Lemma \ref{l-ribexp} (b) and Theorem \ref{t-QrL}, we obtain
\begin{align*}
\left\langle Q(\Pi),\frac{H}{1-tE(y)H}\right\rangle  & =\sum_{L\vDash n}\frac{(1+yt)^{n}}{(1-t)^{n+1}}\left(\frac{(1+y)^{2}t}{(y+t)(1+yt)}\right)^{\lpk(L)}\left(\frac{y+t}{1+yt}\right)^{\des(L)}\left\langle Q(\Pi),r_{L}\right\rangle \\
 & =\frac{(1+yt)^{n}}{(1-t)^{n+1}}\sum_{\pi\in\Pi}\left(\frac{(1+y)^{2}t}{(y+t)(1+yt)}\right)^{\lpk(\pi)}\left(\frac{y+t}{1+yt}\right)^{\des(\pi)}\\
 & =\frac{(1+yt)^{n}}{(1-t)^{n+1}}P^{(\lpk,\des)}\left(\Pi;\frac{(1+y)^{2}t}{(y+t)(1+yt)},\frac{y+t}{1+yt}\right).
\end{align*}
In addition, we have
\begin{align*}
\left\langle Q(\Pi),\frac{H}{1-tE(y)H}\right\rangle  
 & =\left\langle Q(\Pi)[X+1],\frac{1}{1-tE(y)H}\right\rangle  & \text{(by Lemma }\ref{l-spconst})\\
 & =\sum_{k=0}^{\infty}\left\langle Q(\Pi)[X+1],E(y)^{k}H^{k}\right\rangle t^{k}\\
 & =\sum_{k=0}^{\infty}\Theta_{y,k}(Q(\Pi)[X+1])t^{k} & \text{(by Lemma }\ref{l-scalprodh}).
\end{align*}

Equating these two expressions yields part (a). The first two equalities in part
(b) follow from equating the $y=1$ evaluation of these two expressions with
\begin{align*}
\left\langle Q(\Pi),\frac{H}{1-tEH}\right\rangle  & =\left\langle \sum_{\lambda\vdash n}c_{\lambda}p_{\lambda},\sum_{\tau}\frac{p_{\tau}}{z_{\tau}}\frac{B_{o(\tau)}(t)}{(1-t)^{o(\tau)+1}}\right\rangle \\
 & =\sum_{\substack{\lambda\vdash n\\
\tau
}
}\frac{c_{\lambda}}{z_{\tau}}\left\langle p_{\lambda},p_{\tau}\right\rangle \frac{B_{o(\tau)}(t)}{(1-t)^{o(\tau)+1}}\\
 & =\sum_{\lambda\vdash n}c_{\lambda}\frac{B_{o(\lambda)}(t)}{(1-t)^{o(\lambda)+1}},
\end{align*}
which is obtained using Lemma \ref{l-psexp} (b). 
The last equality in (b) comes from the fact that
$\sum_{\lambda\vdash n} c_\lambda w^{o(\lambda)}$ is obtained from $Q(\Pi)$ by setting $p_i=w$ for $i$ odd and $p_i=1$ for $i$ even.
\end{proof}

We state a refinement of Theorem \ref{t-lpkdes} (a) with an additional statistic $\st$, which is proved in a similar way. Define
\[
P^{(\lpk,\des,\st)}(\Pi;y,t,z)\coloneqq\sum_{\pi\in\Pi}y^{\lpk(\pi)}t^{\des(\pi)}z^{\st(\pi)}.
\]
\begin{thm}
\label{t-lpkdesst} Let $\Pi\subseteq\mathfrak{S}_{n}$, with $n\ge1$, and let $\st$
be a permutation statistic. Suppose that $Q^{\st}(\Pi)$ is a symmetric
function \textup{(}in the variables $x_{1},x_{2},\dots$\textup{)}.
Then
\begin{align*}
\frac{(1+yt)^{n}}{(1-t)^{n+1}}P^{(\lpk,\des,\st)}\left(\Pi;\frac{(1+y)^{2}t}{(y+t)(1+yt)},\frac{y+t}{1+yt},z\right) & =\sum_{k=0}^{\infty}\Theta_{y,k}(Q^{\st}(\Pi)[X+1])t^{k}.
\end{align*}
\end{thm}

\subsection{General formula for up-down runs}

We can also provide analogous results for the $\udr$ statistic. Given
a set of permutations $\Pi$ and a permutation statistic $\st$, define
\[
P^{\udr}(\Pi;t)\coloneqq\sum_{\pi\in\Pi}t^{\udr(\pi)}\quad\text{and}\quad P^{(\udr,\st)}(\Pi;t)\coloneqq\sum_{\pi\in\Pi}t^{\udr(\pi)}z^{\st(\pi)}.
\]

\begin{thm}
\label{t-udr} Let $\Pi\subseteq\mathfrak{S}_{n}$, with $n\ge1$, and suppose that
the quasisymmetric generating function $Q(\Pi)$ is a symmetric function
with power sum expansion $Q(\Pi)=\sum_{\lambda\vdash n}c_{\lambda}p_{\lambda}$.
Then we have
\begin{align*}
\frac{(1+t^{2})^{n}}{2(1-t)^{2}(1-t^{2})^{n-1}}P^{\udr}\left(\Pi;\frac{2t}{1+t^{2}}\right) & =\sum_{k=0}^{\infty}\Theta_{1,k}(Q(\Pi))t^{2k}+\sum_{k=0}^{\infty}\Theta_{1,k}(Q(\Pi)[X+1])t^{2k+1}\\
 & =\sum_{\substack{\lambda\vdash n\\
\textup{odd}
}
}c_{\lambda}2^{l(\lambda)}\frac{A_{l(\lambda)}(t^{2})}{(1-t^{2})^{l(\lambda)+1}}+t\sum_{\substack{\lambda\vdash n}
}c_{\lambda}\frac{B_{o(\lambda)}(t^{2})}{(1-t^{2})^{o(\lambda)+1}}\\
&=
\sum_{k=1}^n a_k 2^k \frac{A_k(t^2)}{(1-t^2)^{k+1}} + t\sum_{k=0}^n d_k \frac{B_k(t^2)}{(1-t^2)^{k+1}}
\end{align*}
with $a_k$ and $d_k$ as in Theorems \ref{t-pkdes} and \ref{t-lpkdes}.
\end{thm}

\begin{proof}
We derive three expressions for $\left\langle Q(\Pi),(1+tH)/(1-t^{2}EH)\right\rangle $.
First, from Lemma \ref{l-ribexp} (c) and Theorem \ref{t-QrL} we
have
\begin{align*}
\left\langle Q(\Pi),\frac{1+tH}{1-t^{2}EH}\right\rangle  & =\frac{1}{2(1-t)^{2}}\sum_{L\vDash n}\frac{(1+t^{2})^{n}}{(1-t^{2})^{n-1}}\left(\frac{2t}{1+t^{2}}\right)^{\udr(L)}\left\langle Q(\Pi),r_{L}\right\rangle \\
 & =\frac{(1+t^{2})^{n}}{2(1-t)^{2}(1-t^{2})^{n-1}}\sum_{\pi\in\Pi}\left(\frac{2t}{1+t^{2}}\right)^{\udr(\pi)}\\
 & =\frac{(1+t^{2})^{n}}{2(1-t)^{2}(1-t^{2})^{n-1}}P^{\udr}\left(\Pi;\frac{2t}{1+t^{2}}\right).
\end{align*}
Moreover, following the reasoning of the proof of Theorem \ref{t-lpkdes}
(a), we obtain
\begin{align*}
\left\langle Q(\Pi),\frac{1+tH}{1-t^{2}EH}\right\rangle  & =\left\langle Q(\Pi),\frac{1}{1-t^{2}EH}\right\rangle +\left\langle Q(\Pi),\frac{tH}{1-t^{2}EH}\right\rangle \\
 & =\left\langle Q(\Pi),\frac{1}{1-t^{2}EH}\right\rangle +\left\langle Q(\Pi)[X+1],\frac{t}{1-t^{2}EH}\right\rangle \\
 & =\sum_{k=0}^{\infty}\left\langle Q(\Pi),E^{k}H^{k}\right\rangle t^{2k}+\sum_{k=0}^{\infty}\left\langle Q(\Pi)[X+1],E^{k}H^{k}\right\rangle t^{2k+1}\\
 & =\sum_{k=0}^{\infty}\Theta_{1,k}(Q(\Pi))t^{2k}+\sum_{k=0}^{\infty}\Theta_{1,k}(Q(\Pi)[X+1])t^{2k+1}.
\end{align*}
Finally, Lemma \ref{l-psexp} (c) implies {\allowdisplaybreaks
\begin{align*}
\left\langle Q(\Pi),\frac{1+tH}{1-t^{2}EH}\right\rangle  & =\left\langle \sum_{\lambda\vdash n}c_{\lambda}p_{\lambda},\sum_{\substack{\tau\text{ odd}}
}\frac{p_{\tau}}{z_{\tau}}2^{l(\tau)}\frac{A_{l(\tau)}(t^{2})}{(1-t^{2})^{l(\tau)+1}}\right\rangle \\
 & \qquad\qquad\qquad\qquad\qquad+\left\langle \sum_{\lambda\vdash n}c_{\lambda}p_{\lambda},t\sum_{\tau}\frac{p_{\tau}}{z_{\tau}}\frac{B_{o(\tau)}(t^{2})}{(1-t^{2})^{o(\tau)+1}}\right\rangle \\
 & =\sum_{\lambda\vdash n}\sum_{\substack{\tau\text{ odd}}
}\frac{c_{\lambda}}{z_{\tau}}\left\langle p_{\lambda},p_{\tau}\right\rangle 2^{l(\tau)}\frac{A_{l(\tau)}(t^{2})}{(1-t^{2})^{l(\tau)+1}}+t\sum_{\substack{\lambda\vdash n\\
\tau
}
}\frac{c_{\lambda}}{z_{\mu}}\left\langle p_{\lambda},p_{\tau}\right\rangle \frac{B_{o(\tau)}(t^{2})}{(1-t^{2})^{o(\tau)+1}}\\
 & =\sum_{\substack{\lambda\vdash n\\
\textup{odd}
}
}c_{\lambda}2^{l(\lambda)}\frac{A_{l(\lambda)}(t^{2})}{(1-t^{2})^{l(\lambda)+1}}+t\sum_{\substack{\lambda\vdash n}
}c_{\lambda}\frac{B_{o(\lambda)}(t^{2})}{(1-t^{2})^{o(\lambda)+1}}.
\end{align*}
} The fourth expression follows from the third as in Theorems \ref{t-pkdes} and \ref{t-lpkdes}.
\end{proof}
The proof of the next theorem is similar.
\begin{thm}
\label{t-udrst} Let $\Pi\subseteq\mathfrak{S}_{n}$ and let $\st$
be a permutation statistic. Suppose that $Q^{\st}(\Pi)$ is a symmetric
function \textup{(}in the variables $x_{1},x_{2},\dots$\textup{)}.
Then
\begin{multline*}
\qquad
\frac{(1+t^{2})^{n}}{2(1-t)^{2}(1-t^{2})^{n-1}}P^{(\udr,\st)}\left(\Pi;\frac{2t}{1+t^{2}},z\right)\\
=\sum_{k=0}^{\infty}\Theta_{1,k}(Q^{\st}(\Pi))t^{2k}+\sum_{k=0}^{\infty}\Theta_{1,k}(Q^{\st}(\Pi)[X+1])t^{2k+1}.
\qquad
\end{multline*}
\end{thm}

\section{Cyclic permutations}

We say that a permutation of length $n$ is a \textit{cyclic permutation}
(or a \textit{cycle}) if it has cycle type $(n)$. We denote the set
of cyclic permutations of length $n$ by $\mathfrak{C}_{n}$. 

Recall that the Lyndon symmetric function 
\begin{equation}
L_{n}=\frac{1}{n}\sum_{d\mid n}\mu(d)p_{d}^{n/d}\label{e-lyndonp}
\end{equation}
is the quasisymmetric generating function for $\mathfrak{C}_{n}$.

\subsection{Counting cyclic permutations by peaks and descents}

Define
\[
C_{n}^{(\pk,\des)}(y,t)\coloneqq P^{(\pk,\des)}(\mathfrak{C}_{n};y,t)=\sum_{\pi\in\mathfrak{C}_{n}}y^{\pk(\pi)+1}t^{\des(\pi)+1},
\]
\[
C_{n}^{\pk}(t)\coloneqq P^{\pk}(\mathfrak{C}_{n};t)=\sum_{\pi\in\mathfrak{C}_{n}}t^{\pk(\pi)+1},\quad\text{and}\quad C_{n}(t)\coloneqq A(\mathfrak{C}_{n};t)=\sum_{\pi\in\mathfrak{C}_{n}}t^{\des(\pi)+1}.
\]
Our first result from this section is a formula for the polynomials
$C_{n}^{(\pk,\des)}(y,t)$ in terms of Eulerian polynomials. 
\begin{thm}
\label{t-cycpkdes} Let $n\geq1$. Then
\begin{multline*}
\qquad
C_{n}^{(\pk,\des)}\left(\frac{(1+y)^{2}t}{(y+t)(1+yt)},\frac{y+t}{1+yt}\right)\\
=\frac{1+y}{n(1+yt)^{n+1}}\sum_{d\mid n}\mu(d)(1-(-y)^{d})^{n/d}(1-t)^{n-n/d}A_{n/d}(t).
\qquad
\end{multline*}
\end{thm}

\begin{proof}
The result follows immediately from Theorem \ref{t-pkdes} (a) and
Equation (\ref{e-lyndonp}).
\end{proof}
Although this formula may seem complicated, it allows for easy computation
of the polynomials $C_{n}^{(\pk,\des)}(y,t)$. By inverting this formula (cf.\ the discussion after Lemma \ref{l-ribexp}), we obtain
\begin{align*}
C_{n}^{(\pk,\des)}(y,t) & =\frac{1+u}{n(1+uv)^{n+1}}\sum_{d\mid n}\mu(d)(1-(-u)^{d})^{n/d}(1-v)^{n-n/d}A_{n/d}(v)
\end{align*}
where 
\[
u=\frac{1+t^{2}-2yt-(1-t)\sqrt{(1+t)^{2}-4yt}}{2(1-y)t}
\]
and 
\[
v=\frac{(1+t)^{2}-2yt-(1+t)\sqrt{(1+t)^{2}-4yt}}{2yt}.
\] For example, with
$n=7$, we have 
\begin{align*}
C_{7}^{(\pk,\des)}(y,t) & =\frac{(1+u)}{7(1+uv)^{8}}((1+u)^{7}A_{7}(v)-(1+u^{7})(1-v)^{6}A_{1}(v))\\
 & =\frac{(1+u)}{7(1+uv)^{8}}((1+u)^{7}(v+120v^{2}+1191v^{3}+2416v^{4}+1191v^{5}\\[-8pt]
 & \hspace{2in}+120v^{6}+v^{7})-(1+u^{7})(1-v)^{6}v)\\[2pt]
 & =(y+17y^{2})t^{2}+(2y+64y^{2}+102y^{3})t^{3}+(3y+99y^{2}+207y^{3}+39y^{4})t^{4}\\
 & \qquad\qquad\qquad\qquad\qquad\qquad\qquad\quad+(2y+64y^{2}+102y^{3})t^{5}+(y+17y^{2})t^{6},
\end{align*}
where the last equality was obtained by substituting in the above expressions for $u$ and $v$, and then simplifying using Maple.

Observe that in $C_{7}^{(\pk,\des)}(y,t)$,
the coefficient of $t^{k+1}$ is equal to the coefficient of $t^{n-k}$
for each $k$, indicating that the number of cyclic permutations of
length $7$ with $j$ peaks and $k$ descents is equal to the number
of cyclic permutations of length $7$ with $j$ peaks and $n-1-k$
descents. In fact, this is true for any $n\in\mathbb{P}$ not congruent
to 2 modulo 4. This is not readily apparent from Theorem \ref{t-cycpkdes}, but is a special case of the following result.
\begin{thm}
\label{t-compsym1} Suppose that $\lambda$ is a partition of $n$ with no parts congruent to $2$ modulo $4$ and that every odd part of $\lambda$ occurs only once. Then the number of permutations of cycle type $\lambda$ with $j$ peaks and $k$ descents is equal to the number of permutations of cycle type $\lambda$ with $j$ peaks and $n-1-k$ descents.
\end{thm}

\begin{cor}
If $n$ is not congruent to $2$ modulo $4$, then the number of cyclic permutations of length $n$ with $j$ peaks and $k$ descents is equal to the number of cyclic permutations of length $n$ with $j$ peaks and $n-1-k$ descents.
\end{cor}

To prove Theorem \ref{t-compsym1}, we need the following lemma. Given a descent set $D\subseteq[n-1]$, the \textit{complementary descent set} $D^{c}$ of $D$ is defined by $D^{c}\coloneqq[n-1]\backslash D$.

\begin{lem}
\label{l-compsym1}
Suppose that $\lambda$ is a partition of $n$ with no parts congruent to $2$ modulo $4$ and that every odd part of $\lambda$ occurs only once. Then the number of permutations with cycle type $\lambda$ and descent set $D\subseteq[n-1]$ is equal to the number of permutations with cycle type $\lambda$ and the complementary descent set $D^{c}$.
\end{lem}

Lemma \ref{l-compsym1} is due to Gessel and Reutenauer \cite[Theorem 4.1]{Gessel1993}; a bijective proof was later given by Steinhardt \cite{Steinhardt2010}.

\begin{proof}[Proof of Theorem \ref{t-compsym1}]
Given $\pi\in\mathfrak{S}_{n}$, we say that $i\in\{2,\dots,n\}$ is a \textit{valley} of $\pi$ if $\pi(i-1)>\pi(i)<\pi(i+1)$. It is not difficult to verify that conjugation by the decreasing permutation $n(n-1)\cdots1$ (i.e., the reverse-complement operation) preserves both the descent number and cycle type but toggles between peaks and valleys. In other words, the number of permutations of cycle type $\lambda$ with $j$ peaks and $k$ descents is equal to the number of permutations of cycle type $\lambda$ with $j$ valleys and $k$ descents. The descent number, peak number, and valley number are all descent statistics, and it is easy to see that a permutation with descent set $D$ has $j$ valleys and $k$ descents if and only if a permutation with the complementary descent set $D^{c}$ has $j$ peaks and $n-1-k$ descents. Thus, if $\lambda$ satisfies the conditions of Lemma \ref{l-compsym1}, then the number of permutations of cycle type $\lambda$ with $j$ peaks and $k$ descents is equal to the number of permutations of cycle type $\lambda$ with $j$ peaks and $n-1-k$ descents.
\end{proof}

Let us return to Theorem \ref{t-cycpkdes}. Setting  $y=1$ and $y=0$ yields simpler formulas for counting cyclic permutations by the peak number and by the descent number, respectively. Note that the latter result is the $q=1$ evaluation of \cite[Corollary 6.2]{Gessel1993}.
\begin{cor}\label{c-cycpk}
Let $n\geq1$. Then
\leqnomode
\begin{gather*}\tag{a}
C_{n}^{\pk}\left(\frac{4t}{(1+t)^{2}}\right)=\frac{1}{n(1+t)^{n+1}}
  \sum_{\substack{d\mid n\\d\textup{ odd}}} \mu(d)2^{n/d+1}(1-t)^{n-n/d}A_{n/d}(t)\\
\shortintertext{and}\\[-5pt]
\tag{b}
C_{n}(t)=\frac{1}{n}\sum_{d\mid n}\mu(d)(1-t)^{n-n/d}A_{n/d}(t).
\end{gather*}
\end{cor}

Now, consider the polynomial
\[
P_{n}^{\pk}(t)\coloneqq P^{\pk}(\mathfrak{S}_{n};t)=\sum_{\pi\in\mathfrak{S}_{n}}t^{\pk(\pi)+1}.
\]
giving the distribution of the peak number over all of $\mathfrak{S}_n$. Just as Corollary \ref{c-cycpk} (b) allows one to express the polynomials $C_{n}(t)$ counting cyclic permutations by descents in terms of the polynomials $A_{n}(t)$ counting all permutations by descents, we can use Corollary \ref{c-cycpk} (a) to express $C_{n}^{\pk}(t)$ in terms of the $P_{n}^{\pk}(t)$.
\begin{cor}
\label{c-CpkdesPpkdes}Let $n\geq1$. Then 

\[
C_{n}^{\pk}(t)=\frac{1}{n}\sum_{\substack{d\mid n\\
d\textup{ odd}
}
}\mu(d)(1-t)^{(n-n/d)/2}P_{n/d}^{\pk}(t).
\]
\end{cor}

\begin{proof}
It is known that 
\begin{equation}
A_{k}(t)=\left(\frac{1+t}{2}\right)^{k+1}P_{k}^{\pk}\left(\frac{4t}{(1+t)^{2}}\right)\label{e-Apk}
\end{equation}
for every $k\geq1$; see \cite{Stembridge1997}.\footnote{We will later recover this formula as a special case of Theorem \ref{t-pkdesfixA} (a).} Combining (\ref{e-Apk}) with Corollary \ref{c-cycpk} (a), we obtain 
\begin{align*}
C_{n}^{\pk}\left(\frac{4t}{(1+t)^{2}}\right)=\frac{1}{n}\sum_{\substack{d\mid n\\
d\textup{ odd}
}
}\mu(d)\left(\frac{1-t}{1+t}\right)^{n-n/d}P_{n/d}^{\pk}\left(\frac{4t}{(1+t)^{2}}\right)\label{e-CpkdesPpkdes}
\end{align*}
Inverting this formula and simplifying completes the proof.
\end{proof}

\subsection{Counting cyclic permutations by left peaks}
Let
\[
C_{n}^{\lpk}(t)\coloneqq P^{\lpk}(\mathfrak{C}_{n};t)=\sum_{\pi\in\mathfrak{C}_{n}}t^{\lpk(\pi)}.
\]
We now state a result analogous to Corollary \ref{c-cycpk} (a) but
for the polynomials $C_{n}^{\lpk}(t)$.

\begin{thm}
\label{t-cyclpk} Let $n\geq2$. Then
\begin{align*}
C_{n}^{\lpk}\left(\frac{4t}{(1+t)^{2}}\right) & = 
\frac{B_n(t)-(1-t)^n}{n(1+t)^n}
\end{align*}
if $n$ is a power of $2$; otherwise, 
\[
C_{n}^{\lpk}\left(\frac{4t}{(1+t)^{2}}\right)=\frac{1}{n(1+t)^{n}}\sum_{\substack{d\mid n\\
d\textup{ odd}
}
}\mu(d)(1-t)^{n-n/d}B_{n/d}(t).
\]
\end{thm}

The proof of this theorem requires a M{\"o}bius function identity.
\begin{lem}
\label{l-mu}
\[
\sum_{\substack{d\mid n\\
d\textup{ even}
}
}\mu(d)=\begin{cases}
-1, & \text{if }n=2^{j}\text{ with }j>0,\\
\phantom{-}0, & \text{otherwise.}
\end{cases}
\]
\end{lem}

\begin{proof}
The result is true for $n=1$, so suppose that $n>1$. Then by a fundamental property of the M\"obius function, $\sum_{d\mid n} \mu(d)=0$, so 
\begin{equation*}
\sum_{\substack{d\mid n\\
d\textup{ even}}}\mu(d)
=-\sum_{\substack{d\mid n\\
d\textup{ odd}}}\mu(d).
\end{equation*}
Now let $m$ be the largest odd divisor of $n$. Then 
\begin{equation*}
\sum_{\substack{d\mid n\\
  d\textup{ odd}}}\mu(d)
=\sum_{d\mid m}\mu(d)=
\begin{cases}
1,&\text{if $m=1$},\\
0,&\text{otherwise},
\end{cases}
\end{equation*}
and the result follows.
\end{proof}

\begin{proof}[Proof of Theorem \ref{t-cyclpk}]
By Theorem \ref{t-lpkdes} (b), we have 
\[
C_{n}^{\lpk}\left(\frac{4t}{(1+t)^{2}}\right)=\frac{1}{n(1+t)^{n}}\sum_{\substack{d\mid n\\
d\textup{ odd}
}
}\mu(d)(1-t)^{n-n/d}B_{n/d}(t)+\frac{1}{n}\left(\frac{1-t}{1+t}\right)^{n}\sum_{\substack{d\mid n\\
d\textup{ even}
}
}\mu(d).
\]
Applying Lemma \ref{l-mu} and noting that 1 is the only odd divisor of any power of 2, we obtain the desired result.
\end{proof}

Now, let 
\[
P_{n}^{\lpk}(t)\coloneqq P^{\lpk}(\mathfrak{S}_{n};t)=\sum_{\pi\in\mathfrak{S}_{n}}t^{\lpk(\pi)}.
\]
Our next result expresses the polynomials $C_{n}^{\lpk}(t)$ in terms
of the $P_{n}^{\lpk}(t)$.
\begin{cor}\label{c-ClpkPlpk}
Let $n\geq2$. Then 
\[
C_{n}^{\lpk}(t)=\frac{1}{n}(P_{n}^{\lpk}(t)-(1-t)^{n/2})
\]
if $n$ is a power of $2$; otherwise, 
\[
C_{n}^{\lpk}(t)=\frac{1}{n}\sum_{\substack{d\mid n\\
d\textup{ odd}
}
}\mu(d)(1-t)^{(n-n/d)/2}P_{n/d}^{\lpk}(t).
\]
\end{cor}

\begin{proof}
It is known that 
\[
B_{k}(t)=(1+t)^{k}P_{k}^{\lpk}\left(\frac{4t}{(1+t)^{2}}\right)
\]
for every $k\geq 0$; see \cite[Proposition 4.15]{Petersen2007}.\footnote{We will later recover this formula as a special case of Theorem \ref{t-lpkfixB}.} Combining this
formula with Theorem \ref{t-cyclpk} yields
\begin{align*}
C_{n}^{\lpk}\left(\frac{4t}{(1+t)^{2}}\right) & = 
\frac{1}{n}\left(P_{n}^{\lpk}\left(\frac{4t}{(1+t)^{2}}\right)-\left(\frac{1-t}{1+t}\right)^n\right)
\end{align*}
if $n$ is a power of 2, and
\[
C_{n}^{\lpk}\left(\frac{4t}{(1+t)^{2}}\right)=\frac{1}{n}\sum_{\substack{d\mid n\\
d\textup{ odd}
}
}\mu(d)\left(\frac{1-t}{1+t}\right)^{n-n/d}P_{n/d}^{\lpk}\left(\frac{4t}{(1+t)^{2}}\right)
\]
otherwise. Inverting these formulas and simplifying yields the desired
result.
\end{proof}

\subsection{Counting cyclic permutations by up-down runs}

Finally, we state the analogous result for the polynomials
\[
C_{n}^{\udr}(t)\coloneqq P^{\udr}(\mathfrak{C}_{n};t)=\sum_{\pi\in\mathfrak{C}_{n}}t^{\udr(\pi)},
\]
which is proved in the same way but using Theorem \ref{t-udr}.
\begin{thm}
Let $n\geq2$. Then 
\begin{align*}
C_{n}^{\udr}\left(\frac{2t}{1+t^{2}}\right) & =\frac{2}{n(1+t^{2})^{n}(1+t)^{2}}(2^{n}A_{n}(t^{2})+tB_{n}(t^{2})-t(1-t^{2})^{n})
\end{align*}
if $n$ is a power of $2$; otherwise, 
\begin{align*}
C_{n}^{\udr}\left(\frac{2t}{1+t^{2}}\right) & =\frac{2}{n(1+t^{2})^{n}(1+t)^{2}}\sum_{\substack{d\mid n\\
d\textup{ odd}
}
}\mu(d)(1-t^{2})^{n-n/d}(2^{n/d}A_{n/d}(t^{2})+tB_{n/d}(t^{2})).
\end{align*}
\end{thm}

\section{Involutions}

A permutation $\pi$ is called an \textit{involution} if $\pi^{2}$
is the identity permutation, or equivalently, if $\pi$ has no cycles of length greater than 2. We denote the set of involutions of length $n$ by $\mathfrak{I}_{n}$. 

The quasisymmetric generating function for all involutions weighted
by length and number of fixed points is known to be
\begin{align*}
Q_{I} & \coloneqq\sum_{n=0}^{\infty}Q^{\fix}(\mathfrak{I}_{n})x^{n}=\prod_{i}\frac{1}{1-zxx_{i}}\prod_{i<j}\frac{1}{1-x^{2}x_{i}x_{j}};
\end{align*}
see \cite[Equation (7.1)]{Gessel1993}.

\subsection{Counting involutions by peaks, descents, and fixed points}

Define
\[
I_{n}^{(\pk,\des,\fix)}(y,t,z)\coloneqq P^{(\pk,\des,\fix)}(\mathfrak{I}_{n};y,t,z)=\sum_{\pi\in\mathfrak{I}_{n}}y^{\pk(\pi)+1}t^{\des(\pi)+1}z^{\fix(\pi)};
\]
this polynomial encodes the joint distribution of the peak number,
descent number, and number of fixed points over involutions of length
$n$. Our first result of this section is a formula for the generating
function of these polynomials.
\begin{thm}
\label{t-Ipkdesfix} 
\begin{multline*}
\quad
\frac{1}{1-t}+\frac{1}{1+y}\sum_{n=1}^{\infty}\left(\frac{1+yt}{1-t}\right)^{n+1}I_{n}^{(\pk,\des,\fix)}\left(\frac{(1+y)^{2}t}{(y+t)(1+yt)},\frac{y+t}{1+yt},z\right)x^{n}\\
=\sum_{k=0}^{\infty}\frac{(1+zxy)^{k}(1+x^{2}y)^{k^{2}}t^{k}}{(1-zx)^{k}(1-x^{2})^{{k \choose 2}}(1-x^{2}y^{2})^{{k+1 \choose 2}}}.
\quad
\end{multline*}
\end{thm}

To prove Theorem \ref{t-Ipkdesfix}, we will need to evaluate the
expression $Q_{I}[k(1-\alpha)]$. In doing so, we shall first give
an alternate description of the quasisymmetric generating function
$Q_{I}$ as a plethystic substitution of $H$.
\begin{lem}
\label{l-Qinvalt} $Q_{I}=H[zxe_{1}+x^{2}e_{2}]$
\end{lem}

\begin{proof}
This follows directly from Theorem \ref{t-monic}.
\end{proof}
\begin{lem}
\label{l-QIpls}
\[
Q_{I}[k(1-\alpha)]=\frac{(1-zx\alpha)^{k}(1-x^{2}\alpha)^{k^{2}}}{(1-zx)^{k}(1-x^{2})^{{k \choose 2}}(1-x^{2}\alpha^{2})^{{k+1 \choose 2}}}.
\]
\end{lem}

\begin{proof}
First, observe that
\begin{align}
(zxe_{1}+x^{2}e_{2})[k(1-\alpha)] & =\Big(zxp_{1}+\frac{x^{2}}{2}(p_{1}^{2}-p_{2})\Big)[k(1-\alpha)]\nonumber \\
 & =kzx(1-\alpha)+\frac{x^{2}}{2}(k^{2}(1-\alpha)^{2}-k(1-\alpha^{2}))\nonumber \\
 & =kzx-kzx\alpha+{k \choose 2}x^{2}+{k+1 \choose 2}x^{2}\alpha^{2}-k^{2}x^{2}\alpha.\label{e-invpls}
\end{align}
Thus, we have
\begin{align*}
Q_{I}[k(1-\alpha)] & =H[zxe_{1}+x^{2}e_{2}][k(1-\alpha)] && \text{(by Lemma }\ref{l-Qinvalt})\\
 & =H\left[kzx-kzx\alpha+{k \choose 2}x^{2}+{k+1 \choose 2}x^{2}\alpha^{2}-k^{2}x^{2}\alpha\right] && \text{(by (}\ref{e-invpls}\text{))}\\
 & =\frac{H[zx]^{k}H[x^{2}]^{{k \choose 2}}H[x^{2}\alpha^{2}]^{{k+1 \choose 2}}}{H[zx\alpha]^{k}H[x^{2}\alpha]^{k^{2}}} && \text{(by Lemma }\ref{l-Hps})\\
 & =\frac{(1-zx\alpha)^{k}(1-x^{2}\alpha)^{k^{2}}}{(1-zx)^{k}(1-x^{2})^{{k \choose 2}}(1-x^{2}\alpha^{2})^{{k+1 \choose 2}}}, && \text{(by Lemma }\ref{l-HE}\text{ (c)})
\end{align*}
thus completing the proof.
\end{proof}
We are now ready to prove Theorem \ref{t-Ipkdesfix}.
\begin{proof}[Proof of Theorem \ref{t-Ipkdesfix}]
 By Lemma \ref{l-QIpls}, we have 
\begin{align}
\sum_{k=0}^{\infty}\frac{(1+zxy)^{k}(1+x^{2}y)^{k^{2}}t^{k}}{(1-zx)^{k}(1-x^{2})^{{k \choose 2}}(1-x^{2}y^{2})^{{k+1 \choose 2}}} & =\sum_{k=0}^{\infty}\left.Q_{I}[k(1-\alpha)]\right|_{\alpha=-y}t^{k}\nonumber \\
 & =\sum_{n=0}^{\infty}\sum_{k=0}^{\infty}\Theta_{y,k}(Q^{\fix}(\mathfrak{I}_{n}))t^{k}x^{n}\nonumber \\
 & =\frac{1}{1-t}+\sum_{n=1}^{\infty}\sum_{k=0}^{\infty}\Theta_{y,k}(Q^{\fix}(\mathfrak{I}_{n}))t^{k}x^{n}.\label{e-Ipkdesfix}
\end{align}
Then we apply Theorem \ref{t-pkdesst} to obtain
\begin{multline*}
\qquad
\sum_{n=1}^{\infty}\sum_{k=0}^{\infty}\Theta_{y,k}(Q^{\fix}(\mathfrak{I}_{n}))t^{k}x^{n}\\
=\frac{1}{1+y}\sum_{n=1}^{\infty}\left(\frac{1+yt}{1-t}\right)^{n+1}I_{n}^{(\pk,\des,\fix)}\left(\frac{(1+y)^{2}t}{(y+t)(1+yt)},\frac{y+t}{1+yt},z\right)x^{n},
\qquad
\end{multline*}
which combined with (\ref{e-Ipkdesfix}) proves the result.
\end{proof}
Define
\[
I_{n}^{(\pk,\fix)}(t,z)\coloneqq\sum_{\pi\in\mathfrak{I}_{n}}t^{\pk(\pi)+1}z^{\fix(\pi)}\quad\text{and}\quad I_{n}^{(\des,\fix)}(t,z)\coloneqq\sum_{\pi\in\mathfrak{I}_{n}}t^{\des(\pi)+1}z^{\fix(\pi)}.
\]
Specializing Theorem \ref{t-Ipkdesfix} at $y=1$ yields the following
formula for the joint distribution of $\pk$ and $\fix$ over $\mathfrak{I}_{n}$.
\begin{cor}
\begin{align*}
\frac{1}{1-t}+\frac{1}{2}\sum_{n=1}^{\infty}\left(\frac{1+t}{1-t}\right)^{n+1}I_{n}^{(\pk,\fix)}\left(\frac{4t}{(1+t)^{2}},z\right)x^{n} & =\sum_{k=0}^{\infty}\frac{(1+zx)^{k}(1+x^{2})^{k^{2}}t^{k}}{(1-zx)^{k}(1-x^{2})^{{k \choose 2}}(1-x^{2})^{{k+1 \choose 2}}}
\end{align*}
\end{cor}

Specializing at $y=0$, on the other hand, yields a formula for the joint distribution of $\des$ and $\fix$ over $\mathfrak{I}_{n}$,  which is equivalent to  Equation (5.5) of D\'esarm\'enien and Foata \cite{Foata1985} and Equation (7.3) of Gessel\textendash Reutenauer \cite{Gessel1993}. See also Strehl \cite{Strehl1980} and Guo--Zeng \cite{Guo2006}.

\begin{cor}
\begin{align*}
\frac{1}{1-t}+\sum_{n=1}^{\infty}\frac{I_{n}^{(\des,\fix)}(t,z)}{(1-t)^{n+1}}x^{n} & =\sum_{k=0}^{\infty}\frac{t^{k}}{(1-zx)^{k}(1-x^{2})^{{k \choose 2}}}
\end{align*}
\end{cor}

Our next theorem gives formulas relating $I_{n}^{(\pk,\fix)}(t,z)$ and $I_{n}^{(\des,\fix)}(t,z)$ to Eulerian polynomials.

\begin{thm}
\label{t-eul-inv}
Let $n\ge1$. Then
\leqnomode
\begin{align*}
\tag{a}
\frac{1}{2}\left(\frac{1+t}{1-t}\right)^{n+1}I_{n}^{(\pk,\fix)}\left(\frac{4t}{(1+t)^{2}},z\right) & = \sum_{k=1}^n a_{n,k}(z) 2^k \frac{A_k(t)}{(1-t)^{k+1}}
\end{align*}
where 
\begin{equation*}
1+\sum_{n=1}^\infty x^n\sum_{k=1}^n a_{n,k}(z) w^k   
=\left(\frac{1+zx}{1-zx}\right)^{w/2}\left(\frac{1+x^2}{1-x^2}\right)^{w^2/4},
\end{equation*}
and 
\begin{equation*}
\tag{b}
\frac{I_{n}^{(\des,\fix)}(t,z)}{(1-t)^{n+1}} = \sum_{k=1}^n b_{n,k}(z) \frac{A_k(t)}{(1-t)^{k+1}}
\end{equation*}
where
\begin{equation*}
1+\sum_{n=1}^\infty x^n \sum_{k=1}^n b_{n,k}(z) w^k= \frac{1}{(1-zx)^w(1-x^2)^{\binom{w}{2}}}.
\end{equation*}

\end{thm}
\begin{proof}
We first expand $Q_I$ into power sum symmetric functions. By Lemma \ref{l-Qinvalt}, we have
\begin{align}
Q_I &= H[zxe_1+x^2e_2]\notag\\
   &=H[zxp_1+x^2(p_1^2 -p_2)/2]\notag\\
   &=\exp\biggl(\sum_{n=1}^\infty z^nx^n \frac{p_n}{n} 
     + \sum_{n=1}^\infty x^{2n} \Bigl(\frac{p_n^2-p_{2n}}{2n}\Bigr)\biggr)\notag\\
   &=\exp\biggl(\sum_{n\text{ odd}} z^n x^n \frac{p_n}{n} 
     + \sum_{n\text{ even}}(z^n-1) x^n \frac{p_n}{n}
     +\sum_{n=1}^\infty x^{2n} \frac{p_n^2}{2n}\biggr).
\label{e-QI}
\end{align}
(Here $\sum_{n\text{ even}}$ indicates a sum over even positive values of $n$.)
A generalization of Theorem \ref{t-pkdes} (b) (analogous to Theorem \ref{t-pkdesst}) implies the desired formula for $I_{n}^{(\pk,\fix)}(t,z)$ with
\begin{align}
1+\sum_{n=1}^\infty x^n\sum_{k=1}^n a_{n,k}(z) w^k &= 
\exp\biggl(\sum_{n\text{ odd}} z^n x^n \frac{w}{n} 
     +\sum_{n\text{ odd}}x^{2n} \frac{w^2}{2n}\biggr)\label{e-ank}\\
 &=\exp\biggl(\frac{w}{2}\log\left(\frac{1+zx}{1-zx}\right)  
   +\frac{w^2}{4}\log\left(\frac{1+x^2}{1-x^2}\right) 
 \biggr)\notag\\
 &=\left(\frac{1+zx}{1-zx}\right)^{w/2}\left(\frac{1+x^2}{1-x^2}\right)^{w^2/4}.\notag
\end{align}
Similarly, by applying the analogous generalization of Theorem \ref{t-pkdes} (c) gives the formula for $I_{n}^{(\des,\fix)}(t,z)$ with
\begin{align*}
1+\sum_{n=1}^\infty x^n\sum_{k=1}^n b_{n,k}(z) w^k &= 
\exp\biggl(\sum_{n=1}^\infty z^n x^n \frac{w}{n} 
     +\sum_{n=1}^\infty x^{2n} \frac{w^2-w}{2n}\biggr)\\
 &=\exp\biggl(-w\log(1-zx)-\frac{w(w-1)}{2}\log(1-x^2)
 \biggr)\\
 &=\frac{1}{(1-zx)^w(1-x^2)^{\binom{w}{2}}}.\qedhere
\end{align*}
\end{proof}

The polynomials $a_{n,k}(z)$ and the numbers $b_{n,k}(1)$ have simple combinatorial interpretations.\footnote{We do not have a combinatorial interpretation for the polynomials $b_{n,k}(z)$, which have some negative coefficients.} By the exponential formula \cite[Corollary 5.1.9, p.~7]{Stanley2001}, the exponential generating function for permutations with cycles of length $n$ weighted $u_n$ is $\exp\bigl(\sum_{n=1}^\infty u_n {x^n}/{n}\bigr)$.
Thus, from \eqref{e-ank}, we see that $n!\, \sum_k a_{n,k}(z)w^k$ counts permutations in $\mathfrak{S}_n$ with no cycle lengths divisible by 4 in which odd cycles of length $m$ are weighted $z^m w$ and even cycles (which must have lengths congruent to 2 modulo 4) are weighted $w^2$. For $z=1$, we can restate this in a more elegant way. Note that the square of an odd cycle of length $m$ is an odd cycle of length $m$ and the square of an even cycle of length $m$ is a product of two cycles, each of length $m/2$. Thus, if we set $a_{n,k}\coloneqq a_{n,k}(1)$, then $n!\,a_{n,k}$ is the number of permutations $\pi$ in $\mathfrak{S}_n$ with no cycles having length divisible by 4  for which $\pi^2$ has $k$ cycles. Similarly---as can be seen most easily be setting $z=1$ and $p_n=w$ in \eqref{e-QI}---if we set $b_{n,k}\coloneqq b_{n,k}(1)$, we see that $n!\,b_{n,k}$ is the number of permutations $\pi$ in $\mathfrak{S}_n$ for which $\pi^2$ has $k$ cycles. Therefore, by specializing Theorem \ref{t-eul-inv} appropriately, we obtain the following formulas for the polynomials
\[
I_{n}^{\pk}(t)\coloneqq P^{\pk}(\mathfrak{I}_{n};t)=\sum_{\pi\in\mathfrak{I}_{n}}t^{\pk(\pi)+1} \quad \text{and} \quad I_{n}(t)\coloneqq A(\mathfrak{I}_{n};t)=\sum_{\pi\in\mathfrak{I}_{n}}t^{\des(\pi)+1}.
\]
\begin{cor}
\label{c-IpkA}
Let $n\ge1$. Then
\leqnomode
\begin{align*}
\tag{a}
\frac{1}{2}\left(\frac{1+t}{1-t}\right)^{n+1}I_{n}^{\pk}\left(\frac{4t}{(1+t)^{2}}\right)=\sum_{k=1}^{n}a_{n,k}2^{k}\frac{A_{k}(t)}{(1-t)^{k+1}}
\end{align*}
where $n!\,a_{n,k}$ is the number of permutations $\pi \in \mathfrak{S}_{n}$ with no cycles having length divisible by 4 and for which $\pi^{2}$ has $k$ cycles, and
\begin{equation*}
\tag{b}
\frac{I_{n}(t)}{(1-t)^{n+1}}=\sum_{k=1}^{n}b_{n,k}\frac{A_{k}(t)}{(1-t)^{k+1}}
\end{equation*}
where $n!\,b_{n,k}$ is the number of permutations $\pi \in \mathfrak{S}_{n}$ for which $\pi^{2}$ has $k$ cycles.
\end{cor}

A formula equivalent to Corollary \ref{c-IpkA} (b) was given by Athanasiadis \cite[Proposition 2.22]{Athanasiadis2018}. Athanasiadis's proof involves computations with irreducible characters of symmetric groups.

We can use Corollary \ref{c-IpkA} to obtain an analogue of Corollary \ref{c-CpkdesPpkdes} for involutions. The proof is omitted as it is essentially the same as the proof of
Corollary \ref{c-CpkdesPpkdes}.

\begin{cor}
\label{c-IpkPpk}
Let $n\geq1$. Then
\[
I_{n}^{\pk}(t)=\sum_{k=1}^{n}a_{n,k}(1-t)^{(n-k)/2}P_{k}^{\pk}(t)
\]
with $a_{n,k}$ as in Corollary \ref{c-IpkA}.
\end{cor}

We note that for $a_{n,k}$ to be nonzero, $n$ and $k$ must have the same parity, i.e., $(n-k)/2$ must be an integer. Thus the above formula does not contain any square roots.

Next, we derive a formula for the polynomials
\[
I_{n}^{(\pk,\des)}(y,t)\coloneqq P^{(\pk,\des)}(\mathfrak{I}_{n};y,t)=\sum_{\pi\in\mathfrak{I}_{n}}y^{\pk(\pi)+1}t^{\des(\pi)+1}.
\]
Below, we let $\lambda^2=(\lambda^2_1,\lambda^2_2,\dots)$ be the cycle type of $\pi^2$ for any permutation $\pi$ of cycle type $\lambda$. We denote the parts of the partition $\lambda^2$ by $\lambda^2_1, \lambda^2_2,\dots$ (so $\lambda^2_k$ is not the same as $(\lambda_k)^2$).

\begin{thm}
\label{t-IpkdesAx} For $n\geq1$, we have
\begin{multline*}
\qquad\quad
\frac{1}{1+y}\left(\frac{1+yt}{1-t}\right)^{n+1}I_{n}^{(\pk,\des)}\left(\frac{(1+y)^{2}t}{(y+t)(1+yt)},\frac{y+t}{1+yt}\right)\\
=\sum_{\lambda\vdash n}\frac{A_{l(\lambda^2)}(t)}{z_\lambda(1-t)^{l(\lambda^2)+1}}\prod_{k=1}^{l(\lambda^2)}(1-(-y)^{\lambda_k^2}).
\qquad\quad
\end{multline*}
\end{thm}

\begin{proof} First we set $z=1$ in (\ref{e-QI}), which yields
\begin{align*}
Q_I|_{z=1} &=\exp\biggl(\sum_{n\text{ odd}}  x^n \frac{p_n}{n} 
     +\sum_{n=1}^\infty x^{2n} \frac{p_n^2}{2n}\biggr)\\
     &=\exp\biggl(\sum_{n\text{ odd}} x^n \frac{p_n}{n} 
     +\sum_{n\text{ even}} x^{n} \frac{p_{n/2}^2}{n}\biggr).
\end{align*}
By Lemma \ref{l-expsum}, we have
\begin{equation}
\label{e-qdef}
Q(\mathfrak{I}_{n}) = \sum_{\lambda\vdash n}\frac{q_\lambda}{z_\lambda}
\end{equation}
where $q_\lambda=\prod_{k=1}^{l(\lambda)} q_{\lambda_k}$ with 
$q_n=p_n$ if $n$ is odd  and $q_n = p_{n/2}^2$ if $n$ is even. Since the square of a cycle of odd length $n$ is a cycle of length $n$, and the square of a cycle of even length $n$ is a product of two cycles of length $n/2$, it follows that $q_\lambda=p_{\lambda^2}$.
Thus we may rewrite \eqref{e-qdef} as
\begin{equation}
\label{e-QI2}
Q(\mathfrak{I}_{n}) =\sum_{\lambda\vdash n}\frac{p_{\lambda^2}}{z_\lambda}.
\end{equation}
and the desired formula follows from Theorem \ref{t-pkdes} (a) and \eqref{e-QI2}.
\end{proof}

Theorem \ref{t-IpkdesAx} can also be proven using the approach of Athanasiadis (see the proof of \cite[Proposition 2.22]{Athanasiadis2018}). Moreover, we note that by setting $y=1$ and $y=0$, Theorem \ref{t-IpkdesAx} can be used to obtain Corollary \ref{c-IpkA}.

Finally, we note the following symmetry result on the joint distribution of $\pk$ and $\des$ over $\mathfrak{I}_{n}$, similar to Theorem \ref{t-compsym1}. 
\begin{thm}
\label{t-compsym2}The number of involutions of length $n$ with $j$
peaks and $k$ descents is equal to the number of involutions of length
$n$ with $j$ peaks and $n-1-k$ descents.
\end{thm}

The proof of Theorem \ref{t-compsym2} is the same as that of Theorem \ref{t-compsym1}, but uses the next lemma in place of Lemma \ref{l-compsym1}.

\begin{lem}
\label{l-compsym2}The number of involutions of length $n$ with descent set $D\subseteq[n-1]$ is equal to the number of involutions of length $n$ with the complementary descent set $D^{c}$.
\end{lem}

Lemma \ref{l-compsym2} is Theorem 4.2 of Gessel and Reutenauer \cite{Gessel1993}. An unpublished bijective proof was given by Strehl; see Dukes \cite{Dukes2007} and Steinhardt \cite{Steinhardt2010} for subsequent bijective proofs.

\subsection{Counting involutions by left peaks, descents, and fixed points}

We now give an analogue of Theorem \ref{e-Ipkdesfix} for the polynomials
\[
I_{n}^{(\lpk,\des,\fix)}(y,t,z)\coloneqq P^{(\lpk,\des,\fix)}(\mathfrak{I}_{n};y,t,z)=\sum_{\pi\in\mathfrak{I}_{n}}y^{\lpk(\pi)}t^{\des(\pi)}z^{\fix(\pi)}.
\]
\begin{thm}
\label{t-Ilpkdesfix} 
\begin{multline*}
\qquad
\frac{1}{1-t}+\sum_{n=1}^{\infty}\frac{(1+yt)^{n}}{(1-t)^{n+1}}I_{n}^{(\lpk,\des,\fix)}\left(\frac{(1+y)^{2}t}{(y+t)(1+yt)},\frac{y+t}{1+yt},z\right)x^{n}\\
=\sum_{k=0}^{\infty}\frac{(1+zxy)^{k}(1+x^{2}y)^{k^{2}+k}t^{k}}{(1-zx)^{k+1}(1-x^{2})^{{k+1 \choose 2}}(1-x^{2}y^{2})^{{k+1 \choose 2}}}.
\qquad
\end{multline*}
\end{thm}

The proof of Theorem \ref{t-Ilpkdesfix} proceeds in the same way
as the proof for Theorem \ref{t-Ipkdesfix}, but we use Theorem \ref{t-lpkdesst}
in place of Theorem \ref{t-pkdesst} and we use the next lemma in
place of Lemma \ref{l-QIpls}; the details are omitted.
\begin{lem}
\label{l-QIplsp1}
\[
Q_{I}[X+1][k(1-\alpha)]=\frac{(1-zx\alpha)^{k}(1-x^{2}\alpha)^{k^{2}+k}}{(1-zx)^{k+1}(1-x^{2})^{{k+1 \choose 2}}(1-x^{2}\alpha^{2})^{{k+1 \choose 2}}}.
\]
\end{lem}

\begin{proof}
First, observe that
\begin{align*}
(zxe_{1}+x^{2}e_{2})[X+1] & =zxp_{1}[X+1]+\frac{x^{2}}{2}(p_{1}^{2}-p_{2})[X+1]\\
 & =zx(p_{1}+1)+\frac{x^{2}}{2}((p_{1}+1)^{2}-(p_{2}+1))\\
 & =zx(p_{1}+1)+\frac{x^{2}}{2}(p_{1}^{2}+2p_{1}-p_{2}).
\end{align*}
Then, following the proof of Lemma \ref{l-QIpls}, we have 
\begin{align*}
 & (zxe_{1}+x^{2}e_{2})[X+1][k(1-\alpha)]=\Big(zx(p_{1}+1)+\frac{x^{2}}{2}(p_{1}^{2}+2p_{1}-p_{2})\Big)[k(1-\alpha)]\\
 & \qquad\qquad\qquad\qquad\qquad=zx(k(1-\alpha)+1)+\frac{x^{2}}{2}(k^{2}(1-\alpha)^{2}+2k(1-\alpha)-k(1-\alpha^{2}))\\
 & \qquad\qquad\qquad\qquad\qquad=(k+1)zx-kzx\alpha+{k+1 \choose 2}x^{2}+{k+1 \choose 2}x^{2}\alpha^{2}-(k^{2}+k)x^{2}\alpha
\end{align*}
and thus 
\begin{align*}
 & Q_{I}[X+1][k(1-\alpha)]=H[zxe_{1}+x^{2}e_{2}][X+1][k(1-\alpha)]\\
 & \qquad\qquad\qquad\qquad=H\left[(k+1)zx-kzx\alpha+{k+1 \choose 2}x^{2}+{k+1 \choose 2}x^{2}\alpha^{2}-(k^{2}+k)x^{2}\alpha\right]\\
 & \qquad\qquad\qquad\qquad=\frac{H[zx]^{k+1}H[x^{2}]^{{k+1 \choose 2}}H[x^{2}\alpha^{2}]^{{k+1 \choose 2}}}{H[zx\alpha]^{k}H[x^{2}\alpha]^{k^{2}+k}}\\
 & \qquad\qquad\qquad\qquad=\frac{(1-zx\alpha)^{k}(1-x^{2}\alpha)^{k^{2}+k}}{(1-zx)^{k+1}(1-x^{2})^{{k+1 \choose 2}}(1-x^{2}\alpha^{2})^{{k+1 \choose 2}}}.\qedhere
\end{align*}
\end{proof}
Define
\[
I_{n}^{(\lpk,\fix)}(t,z)\coloneqq\sum_{\pi\in\mathfrak{I}_{n}}t^{\lpk(\pi)}z^{\fix(\pi)}.
\]
Specializing Theorem \ref{t-Ilpkdesfix} at $y=1$ yields the following
corollary.
\begin{cor}
\begin{align*}
\frac{1}{1-t}+\sum_{n=1}^{\infty}\frac{(1+t)^{n}}{(1-t)^{n+1}}I_{n}^{(\lpk,\fix)}\left(\frac{4t}{(1+t)^{2}},z\right)x^{n} & =\sum_{k=0}^{\infty}\frac{(1+zx)^{k}(1+x^{2})^{k^{2}+k}t^{k}}{(1-zx)^{k+1}(1-x^{2})^{k^{2}+k}}
\end{align*}
\end{cor}

We also state an analogue of Theorem \ref{t-eul-inv} for left peaks. We omit the proof as it is similar to that of Theorem \ref{t-eul-inv}.

\begin{thm}
\label{t-lpk-fix-B}
Let $n\geq1$. Then
\[
\frac{(1+t)^{n}}{(1-t)^{n+1}}I_{n}^{(\lpk,\fix)}\left(\frac{4t}{(1+t)^{2}},z\right)=
\sum_{k=0}^n d_{n,k}(z) \frac{B_k(t)}{(1-t)^{k+1}}
\]
where
\begin{equation*}
1+\sum_{n=1}^\infty x^n \sum_{k=0}^n d_{n,k}(z) w^k= 
   \frac{1}{(1-x^4)^{1/4}}
  \left( \frac{1+x^2}{1-x^2}\right)^{w^2/4}\left(\frac{1+zx}{1-zx}\right)^{w/2}
  \left(\frac{1-x^2}{1-z^2x^2}\right)^{1/2}.
\end{equation*}
\end{thm}

As with Theorem \ref{t-eul-inv}, we have a simple combinatorial interpretation when $z=1$: if we set $d_{n,k}\coloneqq d_{n,k}(1)$, then $n!\, d_{n,k}$ is the number of permutations $\pi$ in $\mathfrak{S}_n$ for which $\pi^2$ has $k$ odd cycles (i.e., cycles of odd length). Thus we have the following formula for the polynomials 
\[
I_{n}^{\lpk}(t)\coloneqq P^{\lpk}(\mathfrak{I}_{n};t)=\sum_{\pi\in\mathfrak{I}_{n}}t^{\lpk(\pi)}.
\]
\begin{cor}
\label{c-IlpkA}
Let $n\geq1$. Then
\[
\frac{(1+t)^{n}}{(1-t)^{n+1}}I_{n}^{\lpk}\left(\frac{4t}{(1+t)^{2}}\right)=\sum_{k=0}^n d_{n,k} \frac{B_k(t)}{(1-t)^{k+1}}
\]
where $n!\, d_{n,k}$ is the number of permutations $\pi \in \mathfrak{S}_n$ for which $\pi^2$ has $k$ odd cycles.
\end{cor}
Corollary \ref{c-IlpkA} leads to an expression for $I_{n}^{\lpk}(t)$ in terms of the ordinary left peak polynomials $P_{n}^{\lpk}(t)$. The proof is similar to that of Corollary \ref{c-ClpkPlpk}.

\begin{cor}
\label{c-IlpkPlpk}
Let $n\geq1$. Then
\[
I_{n}^{\lpk}(t)=\sum_{k=0}^n d_{n,k}
(1-t)^{(n-k)/2}P_{k}^{\lpk}(t)
\]
with $d_{n,k}$ as in Corollary \ref{c-IlpkA}.
\end{cor}

For $d_{n,k}$ to be nonzero, $n$ and $k$ must have the same parity, so that $(n-k)/2$ must be an integer. Therefore, like the formula in Corollary \ref{c-IpkPpk}, the above formula does not contain any square roots.

\subsection{Counting involutions by up-down runs and fixed points}

We conclude this section with analogous results for counting involutions
by the number of up-down runs. Define
\[
I_{n}^{(\udr,\fix)}(t,z)\coloneqq P^{(\udr,\fix)}(\mathfrak{I}_{n};t,z)=\sum_{\pi\in\mathfrak{I}_{n}}t^{\udr(\pi)}z^{\fix(\pi)}
\]
and
\[
I_{n}^{\udr}(t)\coloneqq P^{\udr}(\mathfrak{I}_{n};t)=\sum_{\pi\in\mathfrak{I}_{n}}t^{\udr(\pi)}.
\]
\begin{thm}
\label{t-Iudrfix}
\begin{multline*}
\quad
\frac{1}{1-t}+\frac{1}{2(1-t)^{2}}\sum_{n=1}^{\infty}\frac{(1+t^{2})^{n}}{(1-t^{2})^{n-1}}I_{n}^{(\udr,\fix)}\left(\frac{2t}{1+t^{2}},z\right)x^{n}\\
=\sum_{k=0}^{\infty}\frac{(1+zx)^{k}(1+x^{2})^{k^{2}}t^{2k}}{(1-zx)^{k}(1-x^{2})^{k^{2}}}\left(1+\frac{(1+x^{2})^{k}t}{(1-zx)(1-x^{2})^{k}}\right).
\quad
\end{multline*}
\end{thm}

\begin{thm}
\label{t-IudrA} Let $n\geq1$. Then
\begin{equation*}
\frac{(1+t^{2})^{n}}{2(1-t)^{2}(1-t^{2})^{n-1}}I_{n}^{\udr}\left(\frac{2t}{1+t^{2}}\right)
  =\sum_{k=1}^n a_{n,k} 2^{k}\frac{A_{k}(t^2)}{(1-t^{2})^{k+1}}
  +t\sum_{k=0}^n d_{n,k}\frac{B_{k}(t^{2})}{(1-t^{2})^{k+1}}
\end{equation*}
with $a_{n,k}$ and $d_{n,k}$ as in Corollaries \ref{c-IpkA} and \ref{c-IlpkA}.
\end{thm}

Theorem \ref{t-Iudrfix} is proved using Theorem \ref{t-udrst}, Lemma \ref{l-QIpls}, and Lemma \ref{l-QIplsp1}, whereas the proof of Theorem
\ref{t-IudrA} uses Theorem \ref{t-udr}. We omit the details.

\section{Fixed points and derangements}
\label{s-fixder}

Let $\mathfrak{D}_{n}$ denote the set of \textit{derangements}\textemdash permutations
with no fixed points\textemdash of length $n$. In this section, we
will prove formulas for the joint distribution of $\pk$, $\des$,
and $\fix$, the joint distribution of $\lpk$, $\des$, and $\fix$,
and the joint distribution of $\udr$ and $\fix$ over all permutations
and then specialize these results to the case of derangements. 

The quasisymmetric generating function for all permutations weighted
by length and number of fixed points is known to be
\begin{align*}
Q & \coloneqq\sum_{n=0}^{\infty}Q^{\fix}(\mathfrak{S}_{n})x^{n}=\frac{H(zx)}{H(x)(1-p_{1}x)};
\end{align*}
see the proof of \cite[Theorem 8.4]{Gessel1993}.

\subsection{Counting permutations by peaks, descents, and fixed points}

Let us define
\[
P_{n}^{(\pk,\des,\fix)}(y,t,z)\coloneqq P^{(\pk,\des,\fix)}(\mathfrak{S}_{n};y,t,z)=\sum_{\pi\in\mathfrak{S}_{n}}y^{\pk(\pi)+1}t^{\des(\pi)+1}z^{\fix(\pi)}
\]
and
\[
D_{n}^{(\pk,\des)}(y,t)\coloneqq P^{(\pk,\des)}(\mathfrak{D}_{n};y,t)=\sum_{\pi\in\mathfrak{D}_{n}}y^{\pk(\pi)+1}t^{\des(\pi)+1}.
\]
Our first theorem of this section provides generating function formulas
for these polynomials.
\begin{thm}
\label{t-pkdesfix} We have
\leqnomode
\begin{multline*}
\tag{a}
\qquad
\frac{1}{1-t}+\frac{1}{1+y}\sum_{n=1}^{\infty}\left(\frac{1+yt}{1-t}\right)^{n+1}P_{n}^{(\pk,\des,\fix)}\left(\frac{(1+y)^{2}t}{(y+t)(1+yt)},\frac{y+t}{1+yt},z\right)x^{n}\\
=\sum_{k=0}^{\infty}\frac{(1+zxy)^{k}}{(1-zx)^{k}}\frac{(1-x)^{k}}{(1+xy)^{k}}\frac{t^{k}}{1-k(1+y)x}
\qquad
\end{multline*}
and
\begin{multline*}
\tag{b}
\qquad
\frac{1}{1-t}+\frac{1}{1+y}\sum_{n=1}^{\infty}\left(\frac{1+yt}{1-t}\right)^{n+1}D_{n}^{(\pk,\des)}\left(\frac{(1+y)^{2}t}{(y+t)(1+yt)},\frac{y+t}{1+yt}\right)x^{n}\\
=\sum_{k=0}^{\infty}\frac{(1-x)^{k}}{(1+xy)^{k}}\frac{t^{k}}{1-k(1+y)x}.
\qquad
\end{multline*}
\end{thm}

\begin{proof}
First, note that 
\begin{align}
\sum_{k=0}^{\infty}\Theta_{y,k}(Q)t^{k} & =\sum_{k=0}^{\infty}\frac{\Theta_{y,k}(H(zx))}{\Theta_{y,k}(H(x))\Theta_{y,k}(1-p_{1}x)}t^{k}\nonumber \\
 & =\sum_{k=0}^{\infty}\frac{(1+zxy)^{k}}{(1-zx)^{k}}\frac{(1-x)^{k}}{(1+xy)^{k}}\frac{t^{k}}{1-k(1+y)x}\label{e-ThetaQ}
\end{align}
by Lemma \ref{l-HE} (d). Thus, we have 
\begin{align}
\sum_{k=0}^{\infty}\frac{(1+zxy)^{k}}{(1-zx)^{k}}\frac{(1-x)^{k}}{(1+xy)^{k}}\frac{t^{k}}{1-k(1+y)x} & =\sum_{k=0}^{\infty}\Theta_{y,k}(Q)t^{k}\nonumber \\
 & =\sum_{n=0}^{\infty}\sum_{k=0}^{\infty}\Theta_{y,k}(Q^{\fix}(\mathfrak{S}_{n}))t^{k}x^{n}\nonumber \\
 & =\frac{1}{1-t}+\sum_{n=1}^{\infty}\sum_{k=0}^{\infty}\Theta_{y,k}(Q^{\fix}(\mathfrak{S}_{n}))t^{k}x^{n}.\label{e-pkdesfix}
\end{align}
Applying Theorem \ref{t-pkdesst}, we obtain 
\begin{align*}
\frac{1}{1+y}\left(\frac{1+yt}{1-t}\right)^{n+1}P_{n}^{(\pk,\des,\fix)}\left(\frac{(1+y)^{2}t}{(y+t)(1+yt)},\frac{y+t}{1+yt},z\right) & =\sum_{k=0}^{\infty}\Theta_{y,k}(Q^{\fix}(\mathfrak{S}_{n}))t^{k};
\end{align*}
combining this with (\ref{e-pkdesfix}) yields part (a). Part (b)
is simply the $z=0$ specialization of part (a), so we are done.
\end{proof}
Now, define 
\[
P_{n}^{(\pk,\fix)}(t,z)\coloneqq\sum_{\pi\in\mathfrak{S}_{n}}t^{\pk(\pi)+1}z^{\fix(\pi)},\quad A_{n}^{\fix}(t,z)\coloneqq\sum_{\pi\in\mathfrak{S}_{n}}t^{\des(\pi)+1}z^{\fix(\pi)},
\]
\[
D_{n}^{\pk}(t)\coloneqq P^{\pk}(\mathfrak{D}_{n};t,z)=\sum_{\pi\in\mathfrak{D}_{n}}t^{\pk(\pi)+1},\quad\text{and}\quad D_{n}(t)\coloneqq A(\mathfrak{D}_{n};t)=\sum_{\pi\in\mathfrak{D}_{n}}t^{\des(\pi)+1}.
\]
We easily obtain the following generating function formulas for these
polynomials by specializing Theorem \ref{t-pkdesfix} appropriately. 
\begin{cor}
\label{c-pkfixdesfix} We have
\leqnomode
\begin{align*}
\tag{a}
\frac{1}{1-t}+\frac{1}{2}\sum_{n=1}^{\infty}\left(\frac{1+t}{1-t}\right)^{n+1}P_{n}^{(\pk,\fix)}\left(\frac{4t}{(1+t)^{2}},z\right)x^{n} & =\sum_{k=0}^{\infty}\frac{(1+zx)^{k}}{(1-zx)^{k}}\frac{(1-x)^{k}}{(1+x)^{k}}\frac{t^{k}}{1-2kx},
\end{align*}
\begin{align*}
\tag{b}
\frac{1}{1-t}+\frac{1}{2}\sum_{n=1}^{\infty}\left(\frac{1+t}{1-t}\right)^{n+1}D_{n}^{\pk}\left(\frac{4t}{(1+t)^{2}}\right)x^{n} & =\sum_{k=0}^{\infty}\frac{(1-x)^{k}}{(1+x)^{k}}\frac{t^{k}}{1-2kx},
\end{align*}
\begin{align*}
\tag{c}
\frac{1}{1-t}+\sum_{n=1}^{\infty}\frac{A_{n}^{\fix}(t,z)}{(1-t)^{n+1}}x^{n} & =\sum_{k=0}^{\infty}\frac{(1-x)^{k}}{(1-zx)^{k}}\frac{t^{k}}{1-kx},
\end{align*}
and
\begin{align*}
\tag{d}
\frac{1}{1-t}+\sum_{n=1}^{\infty}\frac{D_{n}(t)}{(1-t)^{n+1}}x^{n} & =\sum_{k=0}^{\infty}\frac{(1-x)^{k}t^{k}}{1-kx}.
\end{align*}
\end{cor}

Note that part (c) of Corollary \ref{c-pkfixdesfix} is the $q=1$
evaluation of a formula by Gessel and Reutenauer \cite[Equation (8.3)]{Gessel1993}.

We now give formulas which express $P_n^{(\pk,\fix)}(t,z)$ and $A_n^{\fix}(t,z)$ in terms of Eulerian polynomials. We omit the proof as it is similar to the proof of Theorem \ref{t-eul-inv}.
\begin{thm}
\label{t-pkdesfixA}
Let $n\ge1$. Then
\leqnomode
\begin{align*}
\tag{a}
\frac{1}{2}\left(\frac{1+t}{1-t}\right)^{n+1}P_n^{(\pk,\fix)}\left(\frac{4t}{(1+t)^{2}},z\right) & = \sum_{k=1}^n a_{n,k}(z) 2^k \frac{A_k(t)}{(1-t)^{k+1}}
\end{align*}
where 
\begin{equation*}
1+\sum_{n=1}^\infty x^n\sum_{k=1}^n a_{n,k}(z) w^k   
=\frac{1}{1-wx}\left(\frac{1+zx}{1-zx}\right)^{w/2}\left(\frac{1-x}{1+x}\right)^{w/2},
\end{equation*}
and
\begin{equation*}
\tag{b}
\frac{A_{n}^{\fix}(t,z)}{(1-t)^{n+1}} = \sum_{k=1}^n b_{n,k}(z) \frac{A_k(t)}{(1-t)^{k+1}}
\end{equation*}
where
\begin{equation*}
1+\sum_{n=1}^\infty x^n \sum_{k=1}^n b_{n,k}(z) w^k= \frac{1}{1-wx}\left(\frac{1-x}{1-zx}\right)^w. 
\end{equation*}
\end{thm}

Theorem \ref{t-pkdesfixA} can be used to obtain the next corollary; the proof is very similar to that of Corollary \ref{c-CpkdesPpkdes}.

\begin{cor}
\label{c-pkfixPpk}
Let $n\geq1$. Then
\[
P_{n}^{(\pk,\fix)}(t,z)=\sum_{k=1}^{n}a_{n,k}(z)(1-t)^{(n-k)/2}P_{k}^{\pk}(t)
\]
with $a_{n,k}(z)$ as in Theorem \ref{t-pkdesfixA}.
\end{cor}

Like the formulas in Corollaries \ref{c-IpkPpk} and \ref{c-IlpkPlpk}, the above formula does not contain any square roots.

Specializing Theorem \ref{t-pkdesfixA} and Corollary \ref{c-pkfixPpk} at $z=0$ yields analogous results for derangements; we omit these formulas here. In addition, by specializing Theorem \ref{t-pkdesfixA} (a) at $z=1$ and simplifying, we recover Stembridge's \cite{Stembridge1997} formula 
\begin{equation*}
A_{n}(t)=\left(\frac{1+t}{2}\right)^{n+1}P_{n}^{\pk}\left(\frac{4t}{(1+t)^{2}}\right)
\end{equation*}
relating the Eulerian and peak polynomials.

\subsection{Counting permutations by left peaks, descents, and fixed points}

Next, we prove analogous formulas for the polynomials
\[
P_{n}^{(\lpk,\des,\fix)}(y,t,z)\coloneqq P^{(\lpk,\des,\fix)}(\mathfrak{S}_{n};y,t,z)=\sum_{\pi\in\mathfrak{S}_{n}}y^{\lpk(\pi)}t^{\des(\pi)}z^{\fix(\pi)}
\]
and
\[
D_{n}^{(\lpk,\des)}(y,t)\coloneqq P^{(\lpk,\des)}(\mathfrak{D}_{n};y,t)=\sum_{\pi\in\mathfrak{D}_{n}}y^{\lpk(\pi)}t^{\des(\pi)}.
\]
\begin{thm}
\label{t-lpkdesfix} We have
\leqnomode
\begin{multline*}
\tag{a}
\quad
\frac{1}{1-t}+\sum_{n=1}^{\infty}\frac{(1+yt)^{n}}{(1-t)^{n+1}}P_{n}^{(\lpk,\des,\fix)}\left(\frac{(1+y)^{2}t}{(y+t)(1+yt)},\frac{y+t}{1+yt},z\right)x^{n}\\
=\sum_{k=0}^{\infty}\frac{(1+zxy)^{k}}{(1-zx)^{k+1}}\frac{(1-x)^{k+1}}{(1+xy)^{k}}\frac{t^{k}}{1-(k(1+y)+1)x}
\quad
\end{multline*}\\[-15pt]
and
\begin{multline*}
\tag{b}
\qquad
\frac{1}{1-t}+\sum_{n=1}^{\infty}\frac{(1+yt)^{n}}{(1-t)^{n+1}}D_{n}^{(\lpk,\des)}\left(\frac{(1+y)^{2}t}{(y+t)(1+yt)},\frac{y+t}{1+yt}\right)x^{n}\\
=\sum_{k=0}^{\infty}\frac{(1-x)^{k+1}}{(1+xy)^{k}}\frac{t^{k}}{1-(k(1+y)+1)x}.
\qquad
\end{multline*}
\end{thm}

The following lemma is needed for the proof of Theorem \ref{t-lpkdesfix}.

\begin{lem}
\label{l-HEXplus1} Let $\mathsf{m}$ be a monic term. Then
\begin{enumerate}
\item [\normalfont{(a)}] $H(\mathsf{m})[X+1]=H(\mathsf{m})/(1-\mathsf{m})$,
\item [\normalfont{(b)}] $E(\mathsf{m})[X+1]=(1+\mathsf{m})E(\mathsf{m})$.
\end{enumerate}
\end{lem}

\begin{proof}
We have 
\begin{align*}
H(\mathsf{m})[X+1] & =H[\mathsf{m}X+\mathsf{m}] && \text{(by Lemma }\ref{l-Hps}\text{ (b))}\\
 & =H[\mathsf{m}X]H[\mathsf{m}] && \text{(by Lemma }\ref{l-Hps}\text{ (a))}\\
 & =\frac{H(\mathsf{m})}{1-\mathsf{m}}, && \text{(by Lemma }\ref{l-HE}\text{ (a) and (c))}
\end{align*}
which proves (a). To prove (b), first observe that (a) implies 
\[
(H(-\mathsf{m})[X+1])^{-1}=(1+\mathsf{m})H(-\mathsf{m})^{-1}.
\]
Therefore
\begin{align*}
E(\mathsf{m})[X+1] & =H(-\mathsf{m})^{-1}[X+1]\\
 & =(H(-\mathsf{m})[X+1])^{-1}\\
 & =(1+\mathsf{m})H(-\mathsf{m})^{-1}\\
 & =(1+\mathsf{m})E(\mathsf{m})
\end{align*}
and we are done.
\end{proof}

\begin{proof}[Proof of Theorem \ref{t-lpkdesfix}]
By Lemmas \ref{l-HEXplus1} (a) and \ref{l-HE} (d), we have
\begin{align}
\sum_{k=0}^{\infty}\Theta_{y,k}(Q[X+1])t^{k} & =\sum_{k=0}^{\infty}\frac{\Theta_{y,k}(H(zx)[X+1])}{\Theta_{y,k}(H(x)[X+1])\Theta_{y,k}((1-p_{1}x)[X+1])}t^{k}\nonumber \\
 & =\sum_{k=0}^{\infty}\frac{1-x}{1-zx}\frac{\Theta_{y,k}(H(zx))}{\Theta_{y,k}(H(x))\Theta_{y,k}(1-(p_{1}+1)x)}t^{k}\nonumber \\
 & =\sum_{k=0}^{\infty}\frac{(1+zxy)^{k}}{(1-zx)^{k+1}}\frac{(1-x)^{k+1}}{(1+xy)^{k}}\frac{t^{k}}{1-(k(1+y)+1)x}.\label{e-ThetaQp1}
\end{align}
Then
\begin{align*}
 & \sum_{k=0}^{\infty}\frac{(1+zxy)^{k}}{(1-zx)^{k+1}}\frac{(1-x)^{k+1}}{(1+xy)^{k}}\frac{t^{k}}{1-(k(1+y)+1)x}=\sum_{k=0}^{\infty}\Theta_{y,k}(Q[X+1])t^{k}\\
 & \qquad\qquad\qquad\qquad\qquad\qquad\qquad\qquad\qquad=\sum_{n=0}^{\infty}\sum_{k=0}^{\infty}\Theta_{y,k}(Q^{\fix}(\mathfrak{S}_{n})[X+1])t^{k}x^{n}\\
 & \qquad\qquad\qquad\qquad\qquad\qquad\qquad\qquad\qquad=\frac{1}{1-t}+\sum_{n=1}^{\infty}\sum_{k=0}^{\infty}\Theta_{y,k}(Q^{\fix}(\mathfrak{S}_{n})[X+1])t^{k}x^{n},
\end{align*}
which by Theorem \ref{t-lpkdesst} is equal to 
\[
\frac{1}{1-t}+\sum_{n=1}^{\infty}\frac{(1+yt)^{n}}{(1-t)^{n+1}}P_{n}^{(\lpk,\des,\fix)}\left(\frac{(1+y)^{2}t}{(y+t)(1+yt)},\frac{y+t}{1+yt},z\right)x^{n},
\]
thus establishing part (a). Part (b) follows from setting $z=0$ in
part (a).
\end{proof}
Specializing Theorem \ref{t-lpkdesfix} appropriately, we obtain generating
function formulas for the polynomials
\[
P_{n}^{(\lpk,\fix)}(t,z)\coloneqq\sum_{\pi\in\mathfrak{S}_{n}}t^{\lpk(\pi)}z^{\fix(\pi)}\quad\text{and}\quad D_{n}^{\lpk}(t)\coloneqq P^{\lpk}(\mathfrak{D}_{n};t)=\sum_{\pi\in\mathfrak{D}_{n}}t^{\lpk(\pi)}.
\]

\begin{cor}
We have
\leqnomode
\begin{multline*}
\tag{a}
\qquad\quad
\frac{1}{1-t}+\sum_{n=1}^{\infty}\frac{(1+t)^{n}}{(1-t)^{n+1}}P_{n}^{(\lpk,\fix)}\left(\frac{4t}{(1+t)^{2}},z\right)x^{n}\\
=\sum_{k=0}^{\infty}\frac{(1+zx)^{k}}{(1-zx)^{k+1}}\frac{(1-x)^{k+1}}{(1+x)^{k}}\frac{t^{k}}{1-(2k+1)x}
\qquad\quad
\end{multline*}\\[-15pt]
and
\begin{equation*}
\tag{b}
\frac{1}{1-t}+\sum_{n=1}^{\infty}\frac{(1+t)^{n}}{(1-t)^{n+1}}D_{n}^{\lpk}\left(\frac{4t}{(1+t)^{2}}\right)x^{n}  =\sum_{k=0}^{\infty}\frac{(1-x)^{k+1}}{(1+x)^{k}}\frac{t^{k}}{1-(2k+1)x}.
\end{equation*}
\end{cor}

We omit the proofs of the next theorem and its corollary, as they are very similar to the proofs of analogous results from previous sections.

\begin{thm}
\label{t-lpkfixB}
Let $n\ge1$. Then
\begin{equation*}
\frac{(1+t)^{n}}{(1-t)^{n+1}}P_{n}^{(\lpk,\fix)}\left(\frac{4t}{(1+t)^{2}},z\right)=
\sum_{k=0}^n d_{n,k}(z) \frac{B_k(t)}{(1-t)^{k+1}}
\end{equation*}
where
\begin{equation*}
1+\sum_{n=1}^\infty x^n \sum_{k=0}^n d_{n,k}(z) w^k= 
   \frac{1}{1-xw}
   \left(\frac{1-x^2}{1-z^2 x^2}\right)^{1/2}
  \left( \frac{1+zx}{1-zx}\right)^{w/2}
  \left(\frac{1-x}{1+x}\right)^{w/2}.
\end{equation*}
\end{thm}

\begin{cor}
Let $n\geq1$. Then
\[
P_{n}^{(\lpk,\fix)}(t,z)=\sum_{k=0}^{n}d_{n,k}(z)(1-t)^{(n-k)/2}P_{k}^{\lpk}(t)
\]
with $d_{n,k}(z)$ as in Theorem \ref{t-lpkfixB}.
\end{cor}

Like the formulas in Corollaries \ref{c-IpkPpk}, \ref{c-IlpkPlpk}, and \ref{c-pkfixPpk}, the above formula does not contain any square roots.

We note that setting $z=1$ in Theorem \ref{t-lpkfixB} recovers Petersen's formula \cite[Proposition 4.15]{Petersen2007}
\begin{equation*}
B_{n}(t)=(1+t)^{n}P_{n}^{\lpk}\left(\frac{4t}{(1+t)^{2}}\right)
\end{equation*}
relating the type B Eulerian and left peak polynomials. 

\subsection{Counting permutations by up-down runs and fixed points}

Finally, let us present analogous formulas for the polynomials
\[
P_{n}^{(\udr,\fix)}(t,z)\coloneqq\sum_{\pi\in\mathfrak{S}_{n}}t^{\udr(\pi)}z^{\fix(\pi)}\quad\text{and}\quad D_{n}^{\udr}(t)\coloneqq P^{\udr}(\mathfrak{D}_{n};t)=\sum_{\pi\in\mathfrak{D}_{n}}t^{\udr(\pi)}.
\]
The proofs are omitted as they are very similar to that of analogous results from previous sections.
\begin{thm}
\label{t-udrfix}
We have
\leqnomode
\begin{multline*}
\tag{a}
\quad
\frac{1}{1-t}+\sum_{n=1}^{\infty}\frac{(1+t^{2})^{n}}{2(1-t)^{2}(1-t^{2})^{n-1}}P_{n}^{(\udr,\fix)}\left(\frac{2t}{1+t^{2}},z\right)\\
=\sum_{k=0}^{\infty}\frac{(1+zx)^{k}}{(1-zx)^{k}}\frac{(1-x)^{k}}{(1+x)^{k}}t^{2k}\left(\frac{1}{1-2kx}+\frac{1-x}{1-zx}\frac{t}{1-(2k+1)x}\right)
\quad
\end{multline*}\\[-15pt]
and
\begin{multline*}
\tag{b}
\qquad
\frac{1}{1-t}+\sum_{n=1}^{\infty}\frac{(1+t^{2})^{n}}{2(1-t)^{2}(1-t^{2})^{n-1}}D_{n}^{\udr}\left(\frac{2t}{1+t^{2}}\right)\\
=\sum_{k=0}^{\infty}\frac{(1-x)^{k}}{(1+x)^{k}}t^{2k}\left(\frac{1}{1-2kx}+\frac{(1-x)t}{1-(2k+1)x}\right).
\qquad
\end{multline*}
\end{thm}

\begin{thm}
\label{t-udrfixAB}
Let $n\geq1$. Then
\begin{equation*}
\frac{(1+t^{2})^{n}}{2(1-t)^{2}(1-t^{2})^{n-1}}P_{n}^{(\udr,\fix)}\left(\frac{2t}{1+t^{2}}, z \right)
  =\sum_{k=1}^n a_{n,k}(z) 2^{k}\frac{A_{k}(t^2)}{(1-t^{2})^{k+1}}
  +t\sum_{k=0}^n d_{n,k}(z)\frac{B_{k}(t^{2})}{(1-t^{2})^{k+1}}
\end{equation*}
with $a_{n,k}(z)$ and $d_{n,k}(z)$ as in Theorems \ref{t-pkdesfixA} and \ref{t-lpkfixB}.
\end{thm}

Let 
\[
P_{n}^{\udr}(t)\coloneqq P^{\udr}(\mathfrak{S}_{n};t)=\sum_{\pi\in\mathfrak{S}_{n}}t^{\udr(\pi)}.
\]
We note that setting $z=1$ in Theorem \ref{t-udrfixAB} yields
\begin{align*}
2^{n}A_{n}(t^{2})+tB_{n}(t^{2}) & =\frac{(1+t)^{2}(1+t^{2})^{n}}{2}P_{n}^{\udr}\left(\frac{2t}{1+t^{2}}\right),
\end{align*}
which is equivalent to the formula 
\[
\hat{B}_{n}(t)=\frac{(1+t)(1+t^{2})^{n}}{2t}P_{n}^{\udr}\left(\frac{2t}{1+t^{2}}\right)
\]
previously obtained by the second author \cite[Corollary 4.18]{Zhuang2017}. Here, $\hat{B}_{n}(t)$ is the $n$th flag descent polynomial, which encodes the distribution of the flag descent number over the $n$th hyperoctahedral group; see \cite[Section 2.3]{Zhuang2017} for details.

\section{Cycle type}

In Section \ref{s-fixder}, we derived formulas for counting permutations by the number of fixed points jointly with the peak number and descent number, with the left peak number and descent number, and with the number of up-down runs. We will now extend these results by deriving formulas for counting permutations by these statistics along with cycle type. 

Given a permutation $\pi$, let $N_{i}(\pi)$ denote the number of $i$-cycles in $\pi$. Similarly, given a partition $\lambda$, let $N_{i}(\lambda)$ denote the number of parts of size $i$ in $\pi$. Recall that the Lyndon symmetric function 
\[
L_{\lambda}=h_{m_{1}}[L_{1}]h_{m_{2}}[L_{2}]\cdots
\]
is the quasisymmetric generating function for the set of permutations with cycle type $\lambda=(1^{m_{1}}2^{m_{2}}\cdots)$. Then 
\begin{align*}
P & \coloneqq1+\sum_{n=1}^{\infty}\sum_{\lambda\vdash n}L_{\lambda}x^{n}\prod_{i=1}^{\infty}z_{i}^{N_{i}(\lambda)}\\
 & =\sum_{m_{1},m_{2},\dots}\prod_{i=1}^{\infty}h_{m_{i}}[L_{i}](z_{i}x^{i})^{m_{i}}\\
 & =\prod_{i=1}^{\infty}\sum_{m_{i}=0}^{\infty}h_{m_{i}}[L_{i}](z_{i}x^{i})^{m_{i}}
\end{align*}
is the quasisymmetric generating function for all permutations refined by cycle type and length.

\subsection{Counting permutations by peaks, descents, and cycle type}

Define
\[
F_{n}^{(\pk,\des)}(y,t,z_{1},z_{2},\dots)\coloneqq\sum_{\pi\in\mathfrak{S}_{n}}y^{\pk(\pi)+1}t^{\des(\pi)+1}\prod_{i=1}^{\infty}z_{i}^{N_{i}(\pi)},
\]
\[
F_{n}^{\pk}(t,z_{1},z_{2},\dots)\coloneqq\sum_{\pi\in\mathfrak{S}_{n}}t^{\pk(\pi)+1}\prod_{i=1}^{\infty}z_{i}^{N_{i}(\pi)},
\]
and 
\[
F_{n}^{\des}(t,z_{1},z_{2},\dots)\coloneqq\sum_{\pi\in\mathfrak{S}_{n}}t^{\des(\pi)+1}\prod_{i=1}^{\infty}z_{i}^{N_{i}(\pi)}.
\]
Let us also define the numbers
\[
f_{i,k}\coloneqq\frac{1}{i}\sum_{d\mid i}\mu(d)k^{i/d}\quad\text{and}\quad g_{i,k}\coloneqq\frac{1}{2i}\sum_{\substack{d\mid i\\
d\:\text{odd}
}
}\mu(d)(2k)^{i/d},
\]
which will appear in our formulas for these cycle type polynomials. It is known that $f_{i,k}$ is the number of primitive necklaces of length $i$ from the alphabet $[k]$ (see \cite{Reutenauer1993}) and that $g_{i,k}$ is the number of nowhere-zero primitive twisted necklaces of length $i$ from the alphabet $\pm[k]\coloneqq\{\pm1,\pm2,\dots,\pm k\}$
(see \cite[Remark 3.2]{Fulman2019}).
\begin{thm}
\label{t-pkdesct} We have
\leqnomode
\begin{multline*}
\tag{a}
\quad\frac{1}{1-t}+\frac{1}{1+y}\sum_{n=1}^{\infty}\left(\frac{1+yt}{1-t}\right)^{n+1}\!F_{n}^{(\pk,\des)}\left(\frac{(1+y)^{2}t}{(y+t)(1+yt)},\frac{y+t}{1+yt},z_{1},z_{2},\dots\right)x^{n}\\
=\sum_{k=0}^{\infty}t^{k}\prod_{i=1}^{\infty}\exp\Big(\sum_{m_{i}=1}^{\infty}\frac{(z_{i}x^{i})^{m_{i}}}{im_{i}}\sum_{d\mid i}\mu(d)(k(1-(-y)^{dm_{i}}))^{i/d}\Big),\quad
\end{multline*}
\[\tag{b}
\frac{1}{1-t}+\frac{1}{2}\sum_{n=1}^{\infty}\left(\frac{1+t}{1-t}\right)^{n+1}\!F_{n}^{\pk}\left(\frac{4t}{(1+t)^{2}},z_{1},z_{2},\dots\right)x^{n}=\sum_{k=0}^{\infty}t^{k}\prod_{i=1}^{\infty}\left(\frac{1+z_{i}x^{i}}{1-z_{i}x^{i}}\right)^{g_{i,k}},
\]
and
\[\tag{c}
\frac{1}{1-t}+\sum_{n=1}^{\infty}\frac{F_{n}^{\des}(t,z_{1},z_{2},\dots)}{(1-t)^{n+1}}x^{n}=\sum_{k=0}^{\infty}t^{k}\prod_{i=1}^{\infty}\left(\frac{1}{1-z_{i}x^{i}}\right)^{f_{i,k}}.
\]
\end{thm}

Part (b) is due to Diaconis, Fulman, and Holmes \cite[Corollary 3.8]{Diaconis2013}
and part (c) to Fulman \cite[Theorem 1]{Fulman1998}.
\begin{proof}
It is easy to show that
\begin{multline}
\quad
\frac{1}{1-t}+\frac{1}{1+y}\sum_{n=1}^{\infty}\left(\frac{1+yt}{1-t}\right)^{n+1}F_{n}^{(\pk,\des)}\left(\frac{(1+y)^{2}t}{(y+t)(1+yt)},\frac{y+t}{1+yt},z_{1},z_{2},\dots\right)x^{n}\\
=\sum_{k=0}^{\infty}\Theta_{y,k}(P)t^{k}\label{e-pkdesz}
\quad
\end{multline}
(cf.~the proofs of Theorems \ref{t-pkdesst} and \ref{t-pkdesfix}
(a)), so we shall need to compute $\Theta_{y,k}(P)$. Since 
\[
L_{i}[k(1-\alpha)]=\frac{1}{i}\sum_{d\mid i}\mu(d)(k(1-\alpha^{d}))^{i/d},
\]
we have 
\begin{align}
P[k(1-\alpha)] & =\prod_{i=1}^{\infty}\sum_{m_{i}=0}^{\infty}h_{m_{i}}[L_{i}][k(1-\alpha)](z_{i}x^{i})^{m_{i}}\nonumber \\
 & =\prod_{i=1}^{\infty}\sum_{m_{i}=0}^{\infty}h_{m_{i}}(z_{i}x^{i})^{m_{i}}\Big[\frac{1}{i}\sum_{d\mid i}\mu(d)(k(1-\alpha^{d}))^{i/d}\Big]\nonumber \\
 & =\prod_{i=1}^{\infty}\exp\Bigg(\sum_{m_{i}=1}^{\infty}\frac{p_{m_{i}}}{m_{i}}\Big[\frac{1}{i}\sum_{d\mid i}\mu(d)(k(1-\alpha^{d}))^{i/d}\Big](z_{i}x^{i})^{m_{i}}\Bigg)\nonumber \\
 & =\prod_{i=1}^{\infty}\exp\Big(\sum_{m_{i}=1}^{\infty}\frac{(z_{i}x^{i})^{m_{i}}}{im_{i}}\sum_{d\mid i}\mu(d)(k(1-\alpha^{dm_{i}}))^{i/d}\Big).\label{e-Ppls}
\end{align}
Combining (\ref{e-pkdesz}) with (\ref{e-Ppls}) yields part (a).

To prove part (b), we set $y=1$ in (\ref{e-pkdesz}) to obtain 
\begin{align}
\frac{1}{1-t}+\frac{1}{2}\sum_{n=1}^{\infty}\left(\frac{1+t}{1-t}\right)^{n+1}F_{n}^{\pk}\left(\frac{4t}{(1+t)^{2}},z_{1},z_{2},\dots\right)x^{n} & =\sum_{k=0}^{\infty}\Theta_{1,k}(P)t^{k}.\label{e-pkz}
\end{align}
Setting $\alpha=-1$ in (\ref{e-Ppls}), we obtain $1-\alpha^{dm_{i}}=2$
if both $d$ and $m_{i}$ are odd and $1-\alpha^{dm_{i}}=0$ otherwise,
so
\begin{align}
\Theta_{1,k}(P) & =\prod_{i=1}^{\infty}\exp\Bigg(\sum_{\substack{m_{i}=1\\
m_{i}\text{ odd}
}
}^{\infty}\frac{(z_{i}x^{i})^{m_{i}}}{im_{i}}\sum_{\substack{d\mid i\\
d\text{ odd}
}
}\mu(d)(2k)^{i/d}\Bigg)\nonumber \\
 & =\prod_{i=1}^{\infty}\exp\Bigg(\sum_{\substack{m_{i}=1\\
m_{i}\text{ odd}
}
}^{\infty}\frac{2(z_{i}x^{i})^{m_{i}}}{m_{i}}\Bigg)^{g_{i,k}}\nonumber \\
 & =\prod_{i=1}^{\infty}\exp\Big(\sum_{m_{i}=1}^{\infty}\frac{(z_{i}x^{i})^{m_{i}}}{im_{i}}-\sum_{m_{i}=1}^{\infty}\frac{(-z_{i}x^{i})^{m_{i}}}{im_{i}}\Big)^{g_{i,k}}\nonumber \\
 & =\prod_{i=1}^{\infty}\exp\left(\log(1-z_{i}x^{i})-\log(1+z_{i}x^{i})\right)^{g_{i,k}}\nonumber \\
 & =\prod_{i=1}^{\infty}\left(\frac{1+z_{i}x^{i}}{1-z_{i}x^{i}}\right)^{g_{i,k}}.\label{e-Ppls-1}
\end{align}
Combining (\ref{e-pkz}) and (\ref{e-Ppls-1}) yields part (b). Part
(c) can be proved similarly to part (b), but we set $y=0$ and $\alpha=0$
instead of $y=1$ and $\alpha=-1$.
\end{proof}

\subsection{Counting permutations by left peaks, descents, and cycle type}

We now prove an analogue of Theorem \ref{t-pkdesct} for left peaks.
Let
\[
F_{n}^{(\lpk,\des)}(y,t,z_{1},z_{2},\dots)\coloneqq\sum_{\pi\in\mathfrak{S}_{n}}y^{\lpk(\pi)}t^{\des(\pi)}\prod_{i=1}^{\infty}z_{i}^{N_{i}(\pi)}
\]
and
\[
F_{n}^{\lpk}(t,z_{1},z_{2},\dots)\coloneqq\sum_{\pi\in\mathfrak{S}_{n}}t^{\lpk(\pi)}\prod_{i=1}^{\infty}z_{i}^{N_{i}(\pi)}.
\]
Also, let
\[
h_{i,k}\coloneqq\begin{cases}
\displaystyle\frac{(1+2k)^{i}-1}{2i}, & \text{if }i=2^{j}\text{ with }j\ge0\\
\displaystyle\frac{1}{2i}\sum_{\substack{d\mid i\\
d\:\text{odd}
}
}\mu(d)(1+2k)^{i/d}, & \text{otherwise.}
\end{cases}
\]
The quantity $h_{i,k}$ has a nice combinatorial interpretation as the number of primitive blinking necklaces of length $i$ from the alphabet $\{0\}\cup\pm [k] = \{0, \pm 1, \dots, \pm k \}$; these were first studied by Reiner \cite{Reiner1993} and have also appeared in work by Fulman \cite{Fulman2000}.

\begin{thm}
\label{t-lpkdesct} We have
\leqnomode
\begin{multline*}
\tag{a}
\quad\frac{1}{1-t}+\sum_{n=1}^{\infty}\frac{(1+yt)^{n}}{(1-t)^{n+1}}F_{n}^{(\lpk,\des)}\left(\frac{(1+y)^{2}t}{(y+t)(1+yt)},\frac{y+t}{1+yt},z_{1},z_{2},\dots\right)x^{n}\\
=\sum_{k=0}^{\infty}t^{k}\prod_{i=1}^{\infty}\exp\Big(\sum_{m_{i}=1}^{\infty}\frac{(z_{i}x^{i})^{m_{i}}}{im_{i}}\sum_{d\mid i}\mu(d)(1+k(1-(-y)^{dm_{i}}))^{i/d}\Big)\quad
\end{multline*}
and 
\[
\tag{b}
\frac{1}{1-t}+\sum_{n=1}^{\infty}\frac{(1+t)^{n}}{(1-t)^{n+1}}F_{n}^{\lpk}\left(\frac{4t}{(1+t)^{2}},z_{1},z_{2},\dots\right)x^{n}=\sum_{k=0}^{\infty}\frac{t^{k}}{1-z_{1}x}\prod_{i=1}^{\infty}\left(\frac{1+z_{i}x^{i}}{1-z_{i}x^{i}}\right)^{h_{i,k}}.
\]
\end{thm}

\begin{proof}
First, it is easy to show that
\begin{multline}
\qquad\frac{1}{1-t}+\sum_{n=1}^{\infty}\frac{(1+yt)^{n}}{(1-t)^{n+1}}F_{n}^{(\lpk,\des)}\left(\frac{(1+y)^{2}t}{(y+t)(1+yt)},\frac{y+t}{1+yt},z_{1},z_{2},\dots\right)x^{n}\\
=\sum_{k=0}^{\infty}\Theta_{y,k}(P[X+1])t^{k},\qquad\label{e-lpkdesz}
\end{multline}
so we proceed by computing $\Theta_{y,k}(P[X+1])$. We have
\begin{align*}
L_{i}[X+1][k(1-\alpha)] & =\frac{1}{i}\sum_{d\mid i}\mu(d)(1+p_{d})^{i/d}[k(1-\alpha)]\\
 & =\frac{1}{i}\sum_{d\mid i}\mu(d)(1+k(1-\alpha^{d}))^{i/d};
\end{align*}
thus{\allowdisplaybreaks
\begin{align}
P[X+1][k(1-\alpha)] & =\prod_{i=1}^{\infty}\sum_{m_{i}=0}^{\infty}h_{m_{i}}[L_{i}][X+1][k(1-\alpha)](z_{i}x^{i})^{m_{i}}\nonumber \\
 & =\prod_{i=1}^{\infty}\sum_{m_{i}=0}^{\infty}h_{m_{i}}(z_{i}x^{i})^{m_{i}}\Big[\frac{1}{i}\sum_{d\mid i}\mu(d)(1+k(1-\alpha^{d}))^{i/d}\Big]\nonumber \\
 & =\prod_{i=1}^{\infty}\exp\Bigg(\sum_{m_{i}=1}^{\infty}\frac{p_{m_{i}}}{m_{i}}\Big[\frac{1}{i}\sum_{d\mid i}\mu(d)(1+k(1-\alpha^{d}))^{i/d}\Big](z_{i}x^{i})^{m_{i}}\Bigg)\nonumber \\
 & =\prod_{i=1}^{\infty}\exp\Big(\sum_{m_{i}=1}^{\infty}\frac{(z_{i}x^{i})^{m_{i}}}{im_{i}}\sum_{d\mid i}\mu(d)(1+k(1-\alpha^{dm_{i}}))^{i/d}\Big).\label{e-P=00005BX+1=00005Dpls}
\end{align}
}Then part (a) follows from (\ref{e-lpkdesz}) and (\ref{e-P=00005BX+1=00005Dpls}).

To prove part (b), we first set $y=1$ in (\ref{e-lpkdesz}) to obtain 
\begin{align}
\frac{1}{1-t}+\sum_{n=1}^{\infty}\frac{(1+t)^{n}}{(1-t)^{n+1}}F_{n}^{\lpk}\left(\frac{4t}{(1+t)^{2}},z_{1},z_{2},\dots\right)x^{n} & =\sum_{k=0}^{\infty}\Theta_{1,k}(P[X+1])t^{k}.\label{e-lpkz}
\end{align}
Setting $\alpha=-1$ in (\ref{e-P=00005BX+1=00005Dpls}), we have
$1+k(1-\alpha^{dm_{i}})=1+2k$ if both $d$ and $m_{i}$ are odd and
$1+k(1-\alpha^{dm_{i}})=1$ otherwise. Hence, we have
\begin{multline*}
\quad\Theta_{1,k}(P[X+1])=\prod_{i=1}^{\infty}\exp\Bigg(\sum_{\substack{m_{i}\ge2\\
m_{i}\text{ even}
}
}\frac{(z_{i}x^{i})^{m_{i}}}{im_{i}}\sum_{d\mid i}\mu(d)\Bigg)\\
\times\exp\Bigg(\sum_{\substack{m_{i}\ge 1\\
m_{i}\text{ odd}
}
}\frac{(z_{i}x^{i})^{m_{i}}}{im_{i}}\Bigg(\sum_{\substack{d\mid i\\
d\text{ even}
}
}\mu(d)+\sum_{\substack{d\mid i\\
d\text{ odd}
}
}\mu(d)(1+2k)^{i/d}\Bigg)\Bigg).\quad
\end{multline*}
To simplify this expression, let us define 
\[
h_{i,k}^{\prime}\coloneqq\begin{cases}
\frac{1}{2}+h_{i,k}, & \text{if }i=1,\\
h_{i,k}, & \text{otherwise,}
\end{cases}
\]
and recall that $\sum_{d\mid i}\mu(d)=1$ if $i=1$ and $\sum_{d\mid i}\mu(d)=0$
otherwise. Applying this fact and Lemma \ref{l-mu}, we obtain {\allowdisplaybreaks
\begin{align}
\Theta_{1,k}(P[X+1]) & =\exp\Bigg(\sum_{\substack{m_{1}\geq2\\
m_{1}\text{ even}
}
}\frac{(z_{1}x)^{m_{1}}}{m_{1}}\Bigg)\prod_{i=1}^{\infty}\exp\Bigg(\sum_{\substack{m_{i}\geq1\\
m_{i}\text{ odd}
}
}\frac{2(z_{i}x^{i})^{m_{i}}}{m_{i}}h_{i,k}^{\prime}\Bigg)\nonumber \\
 & =\exp\left(\frac{1}{2}\log\frac{1}{1-z_{1}^{2}x^{2}}\right)\prod_{i=1}^{\infty}\exp\Bigg(\sum_{\substack{m_{i}\geq1\\
m_{i}\text{ odd}
}
}\frac{2(z_{i}x^{i})^{m_{i}}}{m_{i}}\Bigg)^{h_{i,k}^{\prime}}\nonumber \\
 & =\left(\frac{1}{1-z_{1}^{2}x^{2}}\right)^{\frac{1}{2}}\prod_{i=1}^{\infty}\left(\frac{1+z_{i}x^{i}}{1-z_{i}x^{i}}\right)^{h_{i,k}^{\prime}}\nonumber \\
 & =\left(\frac{1+z_{1}x}{(1-z_{1}^{2}x^{2})(1-z_{1}x)}\right)^{\frac{1}{2}}\prod_{i=1}^{\infty}\left(\frac{1+z_{i}x^{i}}{1-z_{i}x^{i}}\right)^{h_{i,k}}\nonumber \\
 & =\frac{1}{1-z_{1}x}\prod_{i=1}^{\infty}\left(\frac{1+z_{i}x^{i}}{1-z_{i}x^{i}}\right)^{h_{i,k}}.\label{e-P=00005BX+1=00005Dpls-1}
\end{align}
}Then combining (\ref{e-lpkz}) and (\ref{e-P=00005BX+1=00005Dpls-1})
yields part (b).
\end{proof}

We note that Theorem \ref{t-lpkdesct} (b) was recently discovered and proven independently by Fulman and Petersen \cite{Fulman2020}.

\subsection{Counting permutations by up-down runs and cycle type}

We end with the analogous formula for the polynomials
\[
F_{n}^{\udr}(t,z_{1},z_{2},\dots)\coloneqq\sum_{\pi\in\mathfrak{S}_{n}}t^{\udr(\pi)}\prod_{i=1}^{\infty}z_{i}^{N_{i}(\pi)}.
\]
\begin{thm}
\begin{multline*}
\quad\frac{1}{1-t}+\sum_{n=1}^{\infty}\frac{(1+t^{2})^{n}}{2(1-t)^{2}(1-t^{2})^{n-1}}F_{n}^{\udr}\left(\frac{2t}{1+t^{2}},z_{1},z_{2},\dots\right)x^{n}\\
=\sum_{k=0}^{\infty}t^{2k}\prod_{i=1}^{\infty}\left(\frac{1+z_{i}x^{i}}{1-z_{i}x^{i}}\right)^{g_{i,k}}+\sum_{k=0}^{\infty}\frac{t^{2k+1}}{1-z_{1}x}\prod_{i=1}^{\infty}\left(\frac{1+z_{i}x^{i}}{1-z_{i}x^{i}}\right)^{h_{i,k}}\quad
\end{multline*}
\end{thm}

\begin{proof}
It is easy to show that
\begin{multline*}
\quad\frac{1}{1-t}+\sum_{n=1}^{\infty}\frac{(1+t^{2})^{n}}{2(1-t)^{2}(1-t^{2})^{n-1}}F_{n}^{\udr}\left(\frac{2t}{1+t^{2}},z_{1},z_{2},\dots\right)x^{n}\\
=\sum_{k=0}^{\infty}\Theta_{1,k}(P)t^{2k}+\sum_{k=0}^{\infty}\Theta_{1,k}(P[X+1])t^{2k+1}.\quad
\end{multline*}
Then combining this equation with (\ref{e-Ppls-1}) and (\ref{e-P=00005BX+1=00005Dpls-1})
yields the result.
\end{proof}

\section{Conclusion}

In summary, we have derived general formulas for studying the joint distribution of the peak number and descent number, the joint distribution of the left peak number and the descent number, and the distribution of the number of up-down runs over any subset $\Pi\subseteq\mathfrak{S}_{n}$ whose quasisymmetric function $Q(\Pi)$ is symmetric. Furthermore, we have applied these results to produce more concrete formulas for distributions of these statistics over cyclic permutations, involutions, derangements, as well as jointly with the number of fixed points and with cycle type over all permutations.

There is a sizable literature on families of permutations whose quasisymmetric generating functions are symmetric. In addition to conjugacy classes of the symmetric group (i.e., permutations with fixed cycle type), these are known to include inverse descent classes and Knuth classes \cite{Gessel1993}, sets of permutations with fixed inversion number \cite{Adin2015}, and more recently, certain classes of pattern-avoiding permutations \cite{Bloom2020a,Bloom2020,Elizalde2017,Hamaker2020}.\footnote{In fact, most of these families of permutations have quasisymmetric generating functions which are not only symmetric but Schur-positive. Schur-positivity of quasisymmetric generating functions is not needed for our work, but has interesting implications for representation theory.} In ongoing work, the present authors are applying this property for inverse descent classes to study various distributions involving ``inverse descent statistics''. One of our preliminary results is a rederivation of the formula
\[
\frac{\sum_{\pi\in\mathfrak{S}_{n}}s^{\des(\pi)}t^{\des(\pi^{-1})}}{(1-s)^{n+1}(1-t)^{n+1}}=\sum_{j,k=0}^{\infty}{jk+n-1 \choose n}s^{j}t^{k}
\]
originally due to Carlitz, Roselle, and Scoville \cite{Carlitz1966} (see also \cite{Petersen2013}) for the joint distribution of the descent number and ``inverse descent number'' over $\mathfrak{S}_{n}$, and we have proved new formulas for similar ``two-sided'' bidistributions such as that of the peak number and ``inverse peak number'' over $\mathfrak{S}_{n}$. It is also worth using the methods developed in this paper to study other families of permutations with symmetric quasisymmetric generating functions, and the case of pattern avoidance classes appears particularly interesting.

\bigskip{}
\noindent \textbf{Acknowledgements.} We thank two anonymous referees for their thoughtful feedback and for pointing us to several relevant references from the literature. The second author also thanks Justin Troyka for helpful discussions.

\bibliographystyle{amsplain}
\addcontentsline{toc}{section}{\refname}\bibliography{bibliography}

\providecommand\noopsort[1]{}
\providecommand{\bysame}{\leavevmode\hbox to3em{\hrulefill}\thinspace}
\providecommand{\MR}{\relax\ifhmode\unskip\space\fi MR }
\providecommand{\MRhref}[2]{%
  \href{http://www.ams.org/mathscinet-getitem?mr=#1}{#2}
}
\providecommand{\href}[2]{#2}
\begin{thebibliography}{10}

\bibitem{Adin2015}
Ron~M. Adin and Yuval Roichman, \emph{Matrices, characters and descents},
  Linear Algebra Appl. \textbf{469} (2015), 381--418. \MR{3299069}

\bibitem{Athanasiadis2018}
Christos~A. Athanasiadis, \emph{Gamma-positivity in combinatorics and
  geometry}, S\'{e}m. Lothar. Combin. \textbf{77} (2018), Art. B77i, 64 pp.
  \MR{3878174}

\bibitem{Bloom2020a}
Jonathan~S. Bloom, \emph{Symmetric multisets of permutations}, J. Combin.
  Theory Ser. A \textbf{174} (2020), 105255, 25. \MR{4080648}

\bibitem{Bloom2020}
Jonathan~S. Bloom and Bruce~E. Sagan, \emph{Revisiting pattern avoidance and
  quasisymmetric functions}, Ann. Comb. \textbf{24} (2020), no.~2, 337--361.
  \MR{4110402}

\bibitem{carlitz1975}
L.~Carlitz, \emph{A combinatorial property of {$q$}-{E}ulerian numbers}, Amer.
  Math. Monthly \textbf{82} (1975), 51--54. \MR{0366683 (51 \#2930)}

\bibitem{Carlitz1966}
L.~Carlitz, D.~P. Roselle, and R.~A. Scoville, \emph{Permutations and sequences
  with repetitions by number of increases}, J. Combinatorial Theory \textbf{1}
  (1966), 350--374.

\bibitem{Foata1985}
Jacques D\'{e}sarm\'{e}nien and Dominique Foata, \emph{Fonctions
  sym\'{e}triques et s\'{e}ries hyperg\'{e}om\'{e}triques basiques
  multivari\'{e}es}, Bull. Soc. Math. France \textbf{113} (1985), no.~1, 3--22.
  \MR{807824}

\bibitem{Diaconis2013}
Persi Diaconis, Jason Fulman, and Susan Holmes, \emph{Analysis of casino shelf
  shuffling machines}, Ann. Appl. Probab. \textbf{23} (2013), no.~4,
  1692--1720. \MR{3098446}

\bibitem{Diaconis1995}
Persi Diaconis, Michael McGrath, and Jim Pitman, \emph{Riffle shuffles, cycles,
  and descents}, Combinatorica \textbf{15} (1995), no.~1, 11--29.

\bibitem{Dukes2007}
W.~M.~B. Dukes, \emph{Permutation statistics on involutions}, European J.
  Combin. \textbf{28} (2007), no.~1, 186--198. \MR{2261870}

\bibitem{Elizalde2011}
Sergi Elizalde, \emph{Descent sets of cyclic permutations}, Adv. in Appl. Math.
  \textbf{47} (2011), no.~4, 688--709. \MR{2832373}

\bibitem{Elizalde2017}
Sergi Elizalde and Yuval Roichman, \emph{Schur-positive sets of permutations
  via products and grid classes}, J. Algebraic Combin. \textbf{45} (2017),
  no.~2, 363--405. \MR{3604061}

\bibitem{Elizalde2019}
Sergi Elizalde and Justin~M. Troyka, \emph{Exact and asymptotic enumeration of
  cyclic permutations according to descent set}, J. Combin. Theory Ser. A
  \textbf{165} (2019), 360--391. \MR{3944532}

\bibitem{Fulman1998}
Jason Fulman, \emph{The distribution of descents in fixed conjugacy classes of
  the symmetric groups}, J. Combin. Theory Ser. A \textbf{84} (1998), no.~2,
  171--180. \MR{1652841}

\bibitem{Fulman2000}
\bysame, \emph{Semisimple orbits of {L}ie algebras and card-shuffling measures
  on {C}oxeter groups}, J. Algebra \textbf{224} (2000), no.~1, 151--165.
  \MR{1736699}

\bibitem{Fulman2001}
\bysame, \emph{Applications of the {B}rauer complex: card shuffling,
  permutation statistics, and dynamical systems}, J. Algebra \textbf{243}
  (2001), no.~1, 96--122. \MR{1851655}

\bibitem{Fulman2019}
Jason Fulman, Gene~B. Kim, and Sangchul Lee, \emph{Central limit theorem for
  peaks of a random permutation in a fixed conjugacy class of ${S}_n$},
  \href{https://arxiv.org/abs/1902.00978}{\texttt{arXiv:1902.00978 [math.CO]}},
  2019.

\bibitem{Fulman2020}
Jason Fulman and T.~Kyle Petersen, \emph{Shuffling and ${P}$-partitions},
  \href{https://arxiv.org/abs/2004.01659}{\texttt{arXiv:2004.01659 [math.CO]}},
  2020.

\bibitem{Gessel1984}
Ira~M. Gessel, \emph{Multipartite {$P$}-partitions and inner products of skew
  {S}chur functions}, Contemp. Math. \textbf{34} (1984), 289--317.

\bibitem{Gessel2016}
\bysame, \emph{A historical survey of {$P$}-partitions}, The {M}athematical
  {L}egacy of {R}ichard {P}. {S}tanley, Amer. Math. Soc., Providence, RI, 2016,
  pp.~169--188. \MR{3617222}

\bibitem{Gessel2016a}
\bysame, \emph{Lagrange inversion}, J. Combin. Theory Ser. A \textbf{144}
  (2016), 212--249. \MR{3534068}

\bibitem{Gessel1993}
Ira~M. Gessel and Christophe Reutenauer, \emph{Counting permutations with given
  cycle structure and descent set}, J. Combin. Theory Ser. A \textbf{64}
  (1993), no.~2, 189--215.

\bibitem{Gessel2018}
Ira~M. Gessel and Yan Zhuang, \emph{Shuffle-compatible permutation statistics},
  Adv. Math. \textbf{332} (2018), 85--141.

\bibitem{Grinberg2020}
Darij Grinberg and Victor Reiner, \emph{Hopf {A}lgebras in {C}ombinatorics},
  \href{https://arxiv.org/abs/1409.8356v6}{\texttt{arXiv:1409.8356 [math.CO]}},
  2020.

\bibitem{Guo2006}
Victor J.~W. Guo and Jiang Zeng, \emph{The {E}ulerian distribution on
  involutions is indeed unimodal}, J. Combin. Theory Ser. A \textbf{113}
  (2006), no.~6, 1061--1071. \MR{2244134}

\bibitem{Haglund2008}
James Haglund, \emph{The {$q$},{$t$}-{C}atalan {N}umbers and the {S}pace of
  {D}iagonal {H}armonics}, University Lecture Series, vol.~41, American
  Mathematical Society, Providence, RI, 2008.

\bibitem{Hamaker2020}
Zachary Hamaker, Brendan Pawlowski, and Bruce~E. Sagan, \emph{Pattern avoidance
  and quasisymmetric functions}, Algebr. Comb. \textbf{3} (2020), no.~2,
  365--388.

\bibitem{Han2009}
Guo-Niu Han and Guoce Xin, \emph{Permutations with extremal number of fixed
  points}, J. Combin. Theory Ser. A \textbf{116} (2009), no.~2, 449--459.
  \MR{2475027}

\bibitem{Loehr2011}
Nicholas~A. Loehr and Jeffrey~B. Remmel, \emph{A computational and
  combinatorial expos\'{e} of plethystic calculus}, J. Algebraic Combin.
  \textbf{33} (2011), no.~2, 163--198. \MR{2765321}

\bibitem{Macdonald1995}
I.~G. Macdonald, \emph{Symmetric {F}unctions and {H}all polynomials}, second
  ed., Oxford Mathematical Monographs, The Clarendon Press, Oxford University
  Press, New York, 1995, With contributions by A. Zelevinsky. \MR{1354144}

\bibitem{macmahon}
Percy~A. MacMahon, \emph{Combinatory {A}nalysis}, Two volumes (bound as one),
  Chelsea Publishing Co., New York, 1960. Originally published in two volumes
  by Cambridge University Press, 1915--1916.

\bibitem{Moustakas2020}
Vassilis-Dionyssis Moustakas, \emph{Specializations of colored quasisymmetric
  functions and {E}uler-{M}ahonian identities},
  \href{https://arxiv.org/abs/2003.07879}{\texttt{arXiv:2003.07879 [math.CO]}},
  2020.

\bibitem{Petersen2007}
T.~Kyle Petersen, \emph{Enriched ${P}$-partitions and peak algebras}, Adv.
  Math. \textbf{209} (2007), no.~2, 561--610.

\bibitem{Petersen2013}
\bysame, \emph{Two-sided {E}ulerian polynomials via balls in boxes}, Math. Mag.
  \textbf{86} (2013), no.~3, 159--176.

\bibitem{Poirier1998}
St\'{e}phane Poirier, \emph{Cycle type and descent set in wreath products},
  Discrete Math. \textbf{180} (1998), no.~1--3, 315--343. \MR{1603753}

\bibitem{Reiner1993}
Victor Reiner, \emph{Signed permutation statistics and cycle type}, European J.
  Combin. \textbf{14} (1993), no.~6, 569--579. \MR{1248064}

\bibitem{Reutenauer1993}
Christophe Reutenauer, \emph{Free {L}ie {A}lgebras}, London Mathematical
  Society Monographs. New Series, vol.~7, The Clarendon Press, New York, 1993.
  \MR{1231799}

\bibitem{Sagan2001}
Bruce~E. Sagan, \emph{The {S}ymmetric {G}roup: {R}epresentations,
  {C}ombinatorial {A}lgorithms, and {S}ymmetric {F}unctions}, second ed.,
  Graduate Texts in Mathematics, vol. 203, Springer-Verlag, New York, 2001.

\bibitem{Stanley2007}
Richard~P. Stanley, \emph{Alternating permutations and symmetric functions}, J.
  Combin. Theory Ser. A \textbf{114} (2007), no.~3, 436--460.

\bibitem{Stanley2008}
\bysame, \emph{Longest alternating subsequences of permutations}, Michigan
  Math. J. \textbf{57} (2008), 675--687.

\bibitem{Stanley2011}
\bysame, \emph{Enumerative {C}ombinatorics, vol. 1}, 2nd ed., Cambridge
  University Press, {\noopsort{a}}2011.

\bibitem{Stanley2001}
\bysame, \emph{Enumerative {C}ombinatorics, vol. 2}, Cambridge University
  Press, {\noopsort{b}}2001.

\bibitem{Steinhardt2010}
Jacob Steinhardt, \emph{Permutations with ascending and descending blocks},
  Electron. J. Combin. \textbf{17} (2010), no.~1, Research Paper 14, 28 pp.
  \MR{2587745}

\bibitem{Stembridge1997}
John~R. Stembridge, \emph{Enriched ${P}$-partitions}, Trans. Amer. Math. Soc.
  \textbf{349} (1997), no.~2, 763--788.

\bibitem{Strehl1980}
Volker Strehl, \emph{Symmetric {E}ulerian distributions for involutions},
  S\'{e}m. Lothar. Combin. \textbf{1} (1980), 12.

\bibitem{Thibon2001}
Jean-Yves Thibon, \emph{The cycle enumerator of unimodal permutations}, Ann.
  Comb. \textbf{5} (2001), no.~3-4, 493--500. \MR{1897638}

\bibitem{Zhuang2016}
Yan Zhuang, \emph{Counting permutations by runs}, J. Comb. Theory Ser. A
  \textbf{142} (2016), 147--176.

\bibitem{Zhuang2017}
\bysame, \emph{Eulerian polynomials and descent statistics}, Adv. in Appl.
  Math. \textbf{90} (2017), 86--144.

\end{thebibliography}

\end{document}